\documentclass[11pt]{amsart}

\usepackage{amsmath,amssymb,amsthm}
\usepackage{thmtools,enumerate}

\usepackage[english]{babel}
\usepackage[autostyle]{csquotes}

\usepackage{geometry}
\usepackage{tikz-cd}
\usetikzlibrary{babel}

\usepackage{hyperref}
\hypersetup{
	colorlinks=true,
	linkcolor=red,
	urlcolor=teal,
	citecolor=blue}

\theoremstyle{plain}

\newtheorem{thm}{Theorem}[section]
\newtheorem{prop}[thm]{Proposition}
\newtheorem{lem}[thm]{Lemma}
\newtheorem{cor}[thm]{Corollary}

\theoremstyle{definition}

\newtheorem{dfn}[thm]{Definition}
\newtheorem{rem}[thm]{Remark}
\newtheorem{exa}[thm]{Example}
\newtheorem*{conv}{Convention}

\theoremstyle{plain}

\newtheorem{thmA}{Theorem}

\newtheorem{corA}[thmA]{Corollary}

\newcommand{\Z}{\mathbb{Z}}
\newcommand{\Q}{\mathbb{Q}}
\newcommand{\R}{\mathbb{R}}
\newcommand{\C}{\mathbb{C}}
\newcommand{\K}{\mathbb{K}}

\newcommand{\OO}{\mathcal{O}}

\DeclareMathOperator{\codim}{codim}
\DeclareMathOperator{\Exc}{Exc}
\DeclareMathOperator{\mult}{mult}
\DeclareMathOperator{\Supp}{Supp}
\DeclareMathOperator{\NEb}{\overline{\mathrm{NE}}}
\DeclareMathOperator{\proj}{Proj}
\DeclareMathOperator{\spec}{Spec}
\DeclareMathOperator{\Alb}{Alb}
\DeclareMathOperator{\sbs}{\mathbf{B}}

\begin{document}
	
	\title[Remarks on the Existence of Minimal Models]{Remarks on the Existence of Minimal Models \\ of Log Canonical Generalized Pairs}
	
	\author[Nikolaos Tsakanikas]{Nikolaos Tsakanikas}
	\address{Institut de Math\'ematiques (CAG), \'Ecole Polytechnique F\'ed\'erale de Lausanne (EPFL), 1015 Lausanne, Switzerland}
	\email{nikolaos.tsakanikas@epfl.ch}
	\email{tsakanikas@mpim-bonn.mpg.de}
	
	\author[Lingyao Xie]{Lingyao Xie}
	\address{Department of Mathematics, The University of Utah, Salt Lake City, UT 84112, USA}
	\email{lingyao@math.utah.edu}
	
	
	\begin{abstract}
		Given an NQC log canonical generalized pair $(X,B+M)$ whose underlying variety $X$ is not necessarily $\Q$-factorial, we show that one may run a $(K_X+B+M)$-MMP with scaling of an ample divisor which terminates, provided that $(X,B+M)$ has a minimal model in a weaker sense or that $K_X+B+M$ is not pseudo-effective. 
		We also prove the existence of minimal models of pseudo-effective NQC log canonical generalized pairs under various additional assumptions, for instance when the boundary contains an ample divisor.
	\end{abstract}
	
	\maketitle
	
	\begingroup
	\hypersetup{linkcolor=black}
	\setcounter{tocdepth}{1}
	\tableofcontents
	\endgroup

	\section{Introduction}
	
	The theory of generalized pairs was originally introduced by Birkar and Zhang \cite{BZ16} in order to address the so-called effective Iitaka fibration conjecture, but nowadays it has become a central topic in higher-dimensional birational geometry due to its plethora of applications. The survey article \cite{Bir21c} provides an overview of various applications of generalized pairs, while further applications concern the existence of minimal models conjecture \cite{LT22a,LT22b,LX23a,LX23b}, the termination of flips conjecture \cite{HM20,CT23} and the generalized non-vanishing conjecture \cite{HanLiu20,Hash22b,LMPTX23}. For the definition of the fundamental concept of an NQC log canonical generalized pair we refer to Subsection \ref{subsection:g-pairs}.
	
	The Minimal Model Program (MMP) for generalized pairs developed rapidly since the introduction of these geometric objects. Initially, it was established for NQC log canonical generalized pairs whose underlying variety has $\Q$-factorial klt singularities \cite{BZ16,HanLi22}, since various core results, such as the existence of flips, could be reduced to analogous statements for usual pairs under this additional assumption on the underlying variety. Currently, the MMP for generalized pairs works in full generality; namely, the papers \cite{HL23}, \cite{Xie22} and \cite{LX23b} proved, respectively, the Cone theorem, the Contraction theorem and the existence of flips for (not necessarily $\Q$-factorial) NQC log canonical generalized pairs. Therefore, it should now be possible, at least in principle, to remove the $\Q$-factoriality assumption from many already existing results about generalized pairs.
	
	In this paper we are mainly concerned with the problem of the existence of minimal models and Mori fiber spaces of (not necessarily $\Q$-factorial) NQC log canonical generalized pairs. To a certain extent, this paper may be regarded as an extension of the previous works \cite{LT22a,LT22b}, since our first objective is to refine the majority of the results of \cite{LT22b} by removing the assumption that the underlying variety is $\Q$-factorial. Our second goal is to make further progress towards the existence of minimal models for generalized pairs by establishing several new special cases. Our results rely essentially on the substantial recent progress in the MMP for generalized pairs mentioned above, and especially on the Contraction theorem for (not necessarily $\Q$-factorial) NQC log canonical generalized pairs \cite{Xie22}. Our main result is the following theorem.
	
	\begin{thmA}[= Theorem \ref{thm:mainthm}]
		\label{thm:mainthm_intro}
		Let $ (X/Z,B+M) $ be an NQC log canonical generalized pair. Assume that either
		\begin{enumerate}[\normalfont (a)]
			\item $ (X,B+M) $ has a minimal model in the sense of Birkar-Shokurov over $ Z $, or
			
			\item $ K_X+B+M $ is not pseudo-effective over $ Z $.
		\end{enumerate}
		Let $ A $ be an effective $ \R $-Cartier $ \R $-divisor on $ X $ which is ample over $ Z $ such that the NQC generalized pair $ \big( X/Z, (B+A)+M \big) $ is log canonical and the divisor $ K_X + B + A + M $ is nef over $Z$. Then there exists a $ (K_X + B + M) $-MMP over $Z$ with scaling of $A$ that terminates.
		In particular:
		\begin{itemize}
			\item $ (X,B+M) $ has a minimal model in the sense of Birkar-Shokurov over $ Z $ if and only if it has a minimal model over $Z$; 
			
			\item if $ K_X+B+M $ is not pseudo-effective over $ Z $, then $ (X,B+M) $ has a Mori fiber space over $ Z $.
		\end{itemize}
	\end{thmA}
	
	We refer to Subsection \ref{subsection:MM_CM_MFS} for the definitions of the various notions of models of a generalized pair that appear in the above statement. We stress that Theorem \ref{thm:mainthm_intro} improves on \cite[Theorem 1.2]{LT22b} in the aforementioned way and extends \cite[Theorem 1.7]{HH20} to the setting of generalized pairs as well. The three main ingredients for its proof are \cite[Theorem 1.3(1)]{LX23a}, \cite[Theorem 1.5]{Xie22} and a refinement of \cite[Theorem 4.1]{LT22b}.
	
	We emphasize that our main result, Theorem \ref{thm:mainthm_intro}, has numerous applications. First, it enables us to achieve our first goal; namely, the following four results are the desired refinements of certain results that were previously obtained in \cite{LT22a,LT22b}. As in op.\ cit., the phrase \enquote{existence of minimal models for smooth varieties} that appears below means the existence of \emph{relative} minimal models, that is, minimal models of smooth quasi-projective varieties which are projective and whose canonical class is pseudoeffective over another normal quasi-projective variety. 
	
	\begin{thmA}[= Theorem \ref{thm:EMM_implication}]
		\label{thm:EMM_implication_intro}
		The existence of minimal models for smooth varieties of dimension $n$ implies the existence of minimal models for NQC log canonical generalized pairs of dimension $n$.
	\end{thmA}
	
	\begin{thmA}[= Theorem \ref{thm:MM_uniruled}]
		\label{thm:MM_uniruled_intro}
		Assume the existence of minimal models for smooth varieties of dimension $n-1$.
		
		Let $ (X/Z,B+M) $ be an NQC log canonical generalized pair of dimension $ n $ such that $ K_X+B+M $ is pseudo-effective over $ Z $. If a general fiber of the morphism $ X \to Z $ is uniruled, then $(X,B+M)$ has a minimal model over $Z$.
	\end{thmA}
	
	\begin{thmA}[= Theorem \ref{thm:MM_NQC_WZD}]
		\label{thm:MM_NQC_WZD_intro}
		Assume the existence of minimal models for smooth varieties of dimension $n-1$.
		
		Let $ (X/Z,B+M) $ be an NQC log canonical generalized pair of dimension $ n $ such that $ K_X+B+M$ is pseudo-effective over $ Z $. The following are equivalent:
		\begin{enumerate}[\normalfont (i)]
			\item $ (X,B+M) $ admits an NQC weak Zariski decomposition over $Z$,
			
			\item $ (X,B+M) $ has a minimal model over $Z$.
		\end{enumerate}
	\end{thmA}
	
	Recall that an NQC generalized pair $(X/Z,B+M)$ admits an NQC weak Zariski decomposition over $Z$ if its canonical class $K_X+B+M$ can be written, birationally and up to numerical equivalence over $Z$, as the sum of an NQC and an effective $\R$-Cartier $\R$-divisor; see Subsection \ref{subsection:NQC_WZD} for the precise definition. The implication (ii) $\implies$ (i) in Theorem \ref{thm:MM_NQC_WZD_intro} is a consequence of the Negativity lemma and does not even require the assumptions in lower dimensions. The essence of Theorem \ref{thm:MM_NQC_WZD_intro} is thus that the converse implication also holds under some mild assumptions in lower dimensions. It refines the previous results \cite[Theorem 1.5]{Bir12b}, \cite[Theorem 1.5]{HanLi22} and \cite[Theorems B and 4.2]{LT22a}.
	
	Since the existence of relative minimal models for smooth varieties of dimension $ n \leq 4 $ was established by \cite[Theorem 5-1-15]{KMM87}, we also deduce the following corollary in low dimensions. Part (i) follows immediately from Theorem \ref{thm:EMM_implication_intro} for $n \leq 4$, while parts (ii) and (iii) are special cases of Theorems \ref{thm:MM_uniruled_intro} and \ref{thm:MM_NQC_WZD_intro}, respectively, for $n=5$.
	
	\begin{corA}[= Corollaries \ref{cor:maincor_I} and \ref{cor:maincor_II}]
		\label{cor:maincor_intro}
		Let $ (X/Z,B+M) $ be an NQC log canonical generalized pair of dimension $ n $ such that $ K_X+B+M$ is pseudo-effective over $ Z $. The following statements hold:
		\begin{enumerate}[\normalfont (i)]
			\item If $n \leq 4$, then $ (X,B+M) $ has a minimal model over $Z$.
			
			\item If $n = 5$ and a general fiber of the morphism $ X\to Z $ is uniruled, then $ (X,B+M) $ has a minimal model over $Z$.
			
			\item If $ n = 5 $ and $ (X,B+M) $ admits an NQC weak Zariski decomposition over $ Z $ (e.g., if $ K_X+B+M $ is effective over $ Z $), then $ (X,B+M) $ has a minimal model over $Z$.
		\end{enumerate}
	\end{corA}
	
	\medskip
	
	Our second objective in this paper is to derive several new results about the existence of minimal models of generalized pairs utilizing Theorem \ref{thm:mainthm_intro}. First and foremost, we show that any NQC log canonical generalized pair whose canonical class is pseudo-effective and whose boundary contains an ample divisor has a \emph{good} minimal model, improving considerably on \cite[Theorem 1.3(2)]{LX23a} and generalizing \cite[Theorem 1.5]{HH20} to the context of generalized pairs.
	
	\begin{thmA}[= Theorem \ref{thm:EGMM_boundary_contains_ample} and Corollary \ref{cor:finite_generation}]
		Let $ \big( X/Z,(B+A)+M \big) $ be an NQC log canonical generalized pair such that $K_X+B+A+M$ is pseudo-effective over $Z$, where $ A $ is an effective $ \R $-Cartier $\R$-divisor which is ample over $ Z $. Then there exists a $(K_X+B+A+M)$-MMP over $Z$ which terminates with a good minimal model of $ \big( X,(B+A)+M \big) $ over $ Z $.
		
		In particular, if $B$, $A$ and $M$ are $\Q$-divisors, then
		\[ R(X/Z,K_X+B+A+M) := \bigoplus_{m \geq 0} \pi_* \OO_X \big( m(K_X+B+A+M) \big) \]
		is a finitely generated $ \OO_Z $-algebra, where $ \pi $ denotes the projective morphism $ X \to Z $.
	\end{thmA}
	
	Moreover, we deal with the problem of the existence of minimal models for NQC log canonical generalized pairs whose underlying variety has maximal Albanese dimension; see Subsection \ref{subsection:mAd} for the definition of this notion for a (smooth or singular) projective variety. Specifically, we first obtain the following generalization of \cite[Theorem 3.4]{Fuj13} to the setting of generalized pairs.
		
	\begin{thmA}[= Theorem \ref{thm:mAd_klt}]
		\label{thm:mAd_klt_intro}
		Let $ (X,B+M) $ be an NQC klt generalized pair. If $ X $ has maximal Albanese dimension, then $ (X,B+M) $ has a minimal model.
	\end{thmA}
	
	It should be mentioned that Theorem \ref{thm:mainthm_intro} does not play any role in the proof of the above theorem. On the other hand, in view of our previous results, one may wonder whether the hypothesis in Theorem \ref{thm:mAd_klt_intro} that $(X,B+M)$ has klt singularities can be replaced by the weaker assumption that $(X,B+M)$ has log canonical singularities. As an indirect application of Theorem \ref{thm:mainthm_intro}, we give an affirmative answer to this question under mild assumptions in lower dimensions.
	
	\begin{thmA}[= Theorem \ref{thm:mAd_lc} and Corollary \ref{cor:mAd_lc_dim5}]
		\label{thm:mAd_lc_intro}
		Assume the existence of minimal models for smooth varieties of dimension $n-1$.
		
		Let $ (X,B+M) $ be an NQC log canonical generalized pair of dimension $ n $ such that $K_X+B+M$ is pseudo-effective. If $ X $ has maximal Albanese dimension, then $ (X,B+M) $ has a minimal model.
		
		In particular, any NQC log canonical generalized pair $(X,B+M)$ of dimension $5$ such that $K_X+B+M$ is pseudo-effective and whose underlying variety $X$ has maximal Albanese dimension has a minimal model.
	\end{thmA}
	
	We discuss now two applications of Theorem \ref{thm:mainthm_intro} concerning the relation between the existence of certain types of Zariski decompositions and the existence of minimal models of NQC log canonical generalized pairs. The first one involving NQC weak Zariski decompositions was already mentioned above; see Theorem \ref{thm:MM_NQC_WZD_intro}. The second one is obtained by considering instead a stronger form of Zariski decomposition, namely the so-called NQC Nakayama-Zariski decomposition, whose definition can be found in Subsection \ref{subsection:NQC_NZD}. More precisely, we deduce the following result, which is valid in the absolute setting and does not require any assumptions in lower dimensions.
	
	\begin{thmA}[= Theorem \ref{thm:MM_NQC_NZD}]
		Let $ (X,B+M) $ be an NQC log canonical generalized pair. Then $ (X,B+M) $ has a minimal model (resp.\ good minimal model) if and only if it admits birationally a Nakayama-Zariski decomposition with NQC (resp.\ semi-ample) positive part.
	\end{thmA}
	
	The above theorem extends \cite[Theorem 1.1]{BH14b} to the setting of generalized pairs, while its conclusion is even stronger. It also has significant consequences with regard to the existence of minimal models conjecture for generalized pairs. Indeed, Theorem \ref{thm:MM_logbigmoduli} and Lemma \ref{lem:Hash20d_3.11} constitute analogues of \cite[Theorem 1.5]{Hash22c} and \cite[Lemma 3.11]{Hash20d}, respectively, in the context of generalized pairs.
	
	\medskip

	Another application of Theorem \ref{thm:mainthm_intro} is the following analogue of \cite[Theorem 1.1]{Gon11} in the context of generalized pairs.
	
	\begin{thmA}[= Theorem \ref{thm:Gon11_Thm1.1}]
		Let $ (X,B+M) $ be an NQC log canonical generalized pair. If $ \kappa_\sigma(X,K_X+B+M) = 0 $, then $ (X,B+M) $ has a minimal model.
	\end{thmA}
	
	Furthermore, we can now extend to the setting of generalized pairs all results from \cite[Subsection 3.2]{Hash20d}, albeit this is often only partially possible; see Section \ref{section:applications_II} for the details. For instance, we obtain the following sufficient conditions for the existence of minimal models of abundant or log abundant NQC log canonical generalized pairs (whose underlying variety is projective); see Subsection \ref{subsection:log_abundant} for the relevant definitions.
	
	\begin{corA}[= Corollaries \ref{cor:Hash20d_3.13} and \ref{cor:Hash20d_3.12}]
		\label{cor:Hash20d_3.13+3.12_intro}
		Let $(X,B+M)$ be an NQC log canonical generalized pair. Assume that $K_X+B+M$ is pseudo-effective and abundant and that all lc centers of $(X,B+M)$ have dimension at most $ 4 $. Then $(X,B+M)$ has a minimal model which is abundant.
		
		In particular, any NQC log canonical generalized pair $(X,B+M)$ of dimension $6$ such that $K_X+B+M$ is pseudo-effective and abundant and $\llcorner B \lrcorner = 0$ has a minimal model which is abundant.
	\end{corA}
	
	\begin{corA}[= Corollary \ref{cor:Hash20d_3.10}]
		\label{cor:Hash20d_3.10_intro}
		Let $(X,B+M)$ be an NQC log canonical generalized pair. Assume that $K_X+B+M$ is pseudo-effective and log abundant with respect to $(X,B+M)$. Assume, moreover, that the stable base locus of $K_X+B+M$ does not contain the center of any divisorial valuation $ P $ over $X$ such that $ a(P,X,B+M) < 0 $. Then $(X,B+M)$ has a minimal model which is log abundant.
	\end{corA}
	
	We conclude the introduction by commenting briefly on the previous two corollaries, beginning with Corollary \ref{cor:Hash20d_3.13+3.12_intro}. First, observe that the $\dim X = 5$ case of Corollary \ref{cor:Hash20d_3.13+3.12_intro} is a special case of Corollary \ref{cor:maincor_intro}(iii) and then the condition $\llcorner B \lrcorner =0$ is actually redundant, whereas the $\dim X = 6$ case of Corollary \ref{cor:Hash20d_3.13+3.12_intro} is new, cf.\ \cite[Corollary 1.6]{Hash20d}. 
	Second, regarding the proof of Corollary \ref{cor:Hash20d_3.13+3.12_intro}, one of its main ingredients is \cite[Theorem 3.14]{Hash22a}. Hashizume's theorem, together with Theorem \ref{thm:mainthm_intro}, imply that an NQC klt generalized pair $(X,B+M)$ such that $K_X+B+M$ is pseudo-effective and \emph{abundant} has a minimal model $(X',B'+M')$, which is also (klt and) abundant. If, moreover, all divisors involved have rational coefficients, then it follows from \cite[Theorem 2]{Chaud23b} that $(X',B'+M')$ is actually a \emph{good} minimal model of $(X,B+M)$; see also \cite[First paragraph in \S 2.2.2]{LX23b}. In particular, klt generalized pairs \emph{of general type} with rational coefficients have good minimal models;  see also Theorem \ref{thm:EMM_general_type} for a more general version of the previous result, which follows readily from \cite{BCHM10}.
	
	Finally, as far as Corollary \ref{cor:Hash20d_3.10_intro} is concerned, prompted by the case of usual pairs, one might expect to establish the existence of a \emph{good} minimal model of $(X,B+M)$ in the setting of Corollary \ref{cor:Hash20d_3.10_intro}. However, as indicated by \cite[Example 2.2]{LX23b}, this need not be true in the context of generalized pairs; in other words, the canonical divisor of a log canonical generalized pair is not necessarily semi-ample even if it is nef and log abundant.
	
	\medskip
	
	\emph{Acknowledgements}: We would like to thank Guodu Chen, Enrica Floris, Christopher Hacon, Jingjun Han, Kenta Hashizume, Vladimir Lazi\'c, Jihao Liu, Zhixin Xie and Sokratis Zikas for useful comments and suggestions, as well as the referee for many valuable comments. 
	Part of this project was completed when N.T.\ visited Enrica Floris at the University of Poitiers in November 2022. He would like to express his gratitude to the members of the Department of Mathematics of the University of Poitiers, especially to Enrica Floris and Sokratis Zikas, for their hospitality. 
	
	N.T.\ is grateful to the Max Planck Institute for Mathematics (MPIM) in Bonn for its hospitality and financial support. N.T.\ also acknowledges support by the ERC starting grant \#804334. L.X.\ is partially supported by the NSF research grants no:\ DMS-1801851, DMS-1952522 and by a grant from the Simons Foundation; Award Number:\ 256202.

	\section{Preliminaries}
	
	Throughout the paper we work over the field $ \C $ of complex numbers. Unless otherwise stated, we assume that varieties are normal and quasi-projective and that a variety $X$ over a variety $Z$, denoted by $X/Z$, is projective over $Z$. We often quote in the paper the Negativity lemma; see \cite[Lemma 3.39(1)]{KM98} and \cite[Lemma 2.3.26]{Fuj17book}.
	A \emph{fibration} is a projective surjective morphism with connected fibers, and a \emph{birational contraction} is a birational map whose inverse does not contract any divisors.
	
	Let $ \pi \colon X\to Z $ be a projective morphism between normal varieties. An $ \R $-Cartier $ \R $-divisor $ D $ on $ X $ is said to be \emph{pseudo-effective over $ Z $} if it is pseudo-effective on a very general fiber of $ \pi $; and \emph{NQC (over $ Z $)} \cite{HanLi22} if it is a non-negative linear combination of $\Q$-Cartier divisors on $X$ which are nef over $ Z $. Two $ \R $-Cartier $ \R $-divisors $ D_1 $ and $ D_2 $ on $ X $ are said to be \emph{$\R$-linearly equivalent over $ Z $}, denoted by $ D_1 \sim_{\R,Z} D_2 $, if there exists an $\R$-Cartier $\R$-divisor $G$ on $Z$ such that $D_1\sim_\R D_2+\pi^*G$; and \emph{numerically equivalent over $ Z $}, denoted by $ D_1 \equiv_Z D_2 $, if it holds that $ D_1\cdot C = D_2 \cdot C $ for any curve $ C $ contained in a fiber of $ \pi $. Finally, we denote by $ N^1(X/Z)_\R $ the $\mathbb{R}$-vector space of relative numerical equivalence classes of $\mathbb{R}$-Cartier divisors on $ X $ over $ Z $
	and by $ \rho(X/Z) $ the \emph{relative Picard number} of $ X $ over $ Z $, i.e., $  \rho(X/Z) := \dim_\R N^1(X/Z)_\R $.

	\subsection{Generalized pairs}
	\label{subsection:g-pairs}
	
	For the standard theory of usual pairs and the Minimal Model Program (MMP) we refer to \cite{KM98,Fuj17book}, while for the recently developed theory of generalized pairs we refer to \cite{BZ16,LT22a,LT22b,Xie22,HL23,LX23a,LX23b} and the relevant references therein.
	We recall now the definitions of generalized pairs and their usual classes of singularities. Afterwards, we briefly discuss some basic results about generalized pairs.
	
	\begin{dfn}
		A \emph{generalized pair}, abbreviated as \emph{g-pair}, consists of 
		\begin{itemize}
			\item a normal variety $ X $, equipped with a projective morphism $ X \to Z $,
			
			\item an effective $ \R $-divisor $ B $ on $X$,
			
			\item a projective birational morphism $ f \colon W \to X $ from a normal variety $ W $ and an $\R$-Cartier $\R$-divisor $ M_W $ on $ W $ which is nef over $ Z $,
		\end{itemize}
		such that the divisor $ K_X + B + M $ is $ \R $-Cartier, where $ M := f_* M_W $. 
		We say that the divisor $ B $ (resp.\ $ M $) is the \emph{boundary part} (resp.\ the \emph{nef part}) of the g-pair, and we call the given g-pair \emph{NQC} if $M_W$ is an NQC divisor (over $ Z $) on $ W $.
	\end{dfn}
	
	We note that the variety $ W $ in the definition may always be chosen as a sufficiently high birational model of $ X $; see \cite[Definition 1.4]{BZ16}. Usually we denote a g-pair as above by $(X/Z,B+M)$, but remember the whole g-pair structure. In a few occasions all divisors involved will be $\Q$-divisors and then we will use the term \emph{$\Q$-g-pair} to refer to such a g-pair. Sometimes we will work exclusively in the absolute setting ($Z = \spec \C$) and then the underlying variety $X$ of any given g-pair $(X,B+M)$ will be projective, even though this will not be mentioned explicitly, since it is implied by our conventions above.
	
	\begin{dfn}
		Let $ (X,B+M) $ be a g-pair with data $ W \overset{f}{\longrightarrow} X \to Z $ and $ M_W $. Let $ E $ be a divisorial valuation over $ X $. We may assume that the \emph{center} $ c_W (E) $ of $E$ on $W$ is a prime divisor on $ W $. If we write 
		\[ K_W + B_W + M_W = f^* ( K_X + B + M ) \]
		for some $ \R $-divisor $ B_W $ on $ W $, then the \emph{discrepancy of $ E $ with respect to $ (X,B+M) $} is defined as
		\[ a(E, X, B+M) :=-\mult_E B_W . \]
		
		We say that the g-pair $ (X,B+M) $ is:
		\begin{itemize}
			\item \emph{klt} if $  a(E,X,B+M) > -1 $ for any divisorial valuation $ E $ over $ X $;
			
			\item \emph{lc} if $  a(E,X,B+M) \geq -1 $ for any divisorial valuation $ E $ over $ X $;
			
			\item \emph{dlt} if it is lc and if there exists an open subset $U \subseteq X$ such that $(U,B|_U)$ is a log smooth pair, and if $ a(E,X,B+M) = -1 $ for some divisorial valuation $ E $ over $X$, then $ c_X(E) \cap U \neq \emptyset $ and $ c_X(E) \cap U $ is an lc center of $(U,B|_U)$.
		\end{itemize}
	\end{dfn}
	
	We highlight that, according to \cite[Theorem 6.1]{Hash22a}, the above definition of dlt singularities, namely \cite[Definition 2.3]{HanLi22}, and the one from \cite[Subsection 2.13(2)]{Bir19} coincide for NQC g-pairs.
	
	We also recall that, given an  lc g-pair $ (X,B+M) $, an irreducible subvariety $ S $ of $ X $ is called an \emph{lc center} of $ (X,B+M) $ if there exists a divisorial valuation $ E $ over $ X $ such that $  c_X(E) = S $ and $ a(E,X,B + M) = -1 $.
	
	\medskip
	
	The next result is \cite[Proposition 3.10]{HanLi22} and will be frequently used in the paper without explicit mention.
	
	\begin{lem}
		Let $ (X,B+M) $ be an lc g-pair with data $ W \overset{f}{\to} X \to Z $ and $ M_W $. Then, after possibly replacing $ f $ with a higher model, there exist a $\Q$-factorial dlt g-pair $(X',B'+M')$ with data $ W \overset{g}{\to} X' \to Z $ and $ M_W $, and a projective birational morphism $ h \colon X' \to X $ such that 
		\[ K_{X'} + B' + M' \sim_\R h^* (K_X + B + M) \quad \text{and} \quad B' = h_*^{-1} B + E , \]
		where $ E $ is the sum of all $ h $-exceptional prime divisors on $ X' $. The g-pair $(X',B'+M')$ is called a \emph{dlt blow-up} of $(X,B+M)$.
	\end{lem}
	
	We now derive an easy corollary of the Negativity lemma, which plays a key role in the paper nonetheless.
	
	\begin{lem}\label{lem:pullback_P+N}
		Let $ (X,B+M) $ be an NQC lc g-pair with data $ W \overset{f}{\longrightarrow} X \to Z $ and $ M_W $. Let $ P $ be the pushforward to $ X $ of an NQC divisor (over $ Z $) on a birational model of $ X $ and let $ N $ be an effective $ \R $-divisor on $ X $ such that $ N+P $ is $ \R $-Cartier. After possibly replacing $f$ with a higher model, we may assume that $W$ is smooth and that there exists an $\R$-divisor $P_W$ on $W$ such that $P_W$ is NQC (over $Z$) and $ f_* P_W = P $. Then we may write
		\[ f^*(P+N) = P_W + f_*^{-1} N + E_W , \]
		where $E_W$ is an effective $f$-exceptional $\R$-divisor on $W$.
	\end{lem}
	
	\begin{proof}
		Since $ f_* (P_W + f_*^{-1} N) = P + N $ by construction, we may write
		\[ f^*(P+N) = P_W + f_*^{-1} N + E_W \]
		for some $f$-exceptional $\R$-Cartier $\R$-divisor $E_W$ on $W$. Since $P_W$ is clearly nef over $X$, we infer that $ - (f_*^{-1} N + E_W) $ is also nef over $X$, and since $N \geq 0$, by the Negativity lemma we obtain $ f_*^{-1} N + E_W \geq 0 $. But $ f_*^{-1} N $ and $ E_W $ have no common components, which implies that $ E_W \geq 0 $, as claimed.
	\end{proof}
	
	Finally, the following result will be often used in the paper without explicit mention, cf.\ \cite[Lemma 2.3]{LT22b}. It is an analogue of \cite[Corollaries 2.35(1) and 2.39(1)]{KM98} in the context of g-pairs. For brevity we only outline its proof below.
	
	\begin{lem}
		Let $ (X/Z,B+M) $ be a g-pair. Let $ P $ be the pushforward to $ X $ of an NQC divisor (over $ Z $) on a birational model of $ X $ and let $ N $ be an effective $ \R $-divisor on $ X $ such that $ N+P $ is $ \R $-Cartier. If the g-pair $ \big( X, (B+N) + (M+P) \big) $ is klt (resp.\ dlt, resp.\ lc), then the g-pair $ (X,B+M) $ is also klt (resp.\ dlt, resp.\ lc).
	\end{lem}
	
	\begin{proof}
		To prove the statement for klt (resp.\ lc) singularities, we argue as in the proof of \cite[Lemma 2.7]{CT23} and we apply Lemma \ref{lem:pullback_P+N} instead of invoking directly the Negativity lemma as in the proof of op.\ cit. To prove the statement for dlt singularities, we argue as in the proof of \cite[Lemma 2.3]{LT22b} and we apply Lemma \ref{lem:pullback_P+N} as explained previously.
	\end{proof}

	\subsection{Minimal models, canonical models and Mori fiber spaces}
	\label{subsection:MM_CM_MFS}
	
	We first recall the definition of (good) minimal models and Mori fiber spaces both \emph{in the usual sense} and \emph{in the sense of Birkar-Shokurov}.
	
	\begin{dfn}\label{dfn:MM_MFS}
		Assume that we have a birational map $\varphi \colon X \dashrightarrow X'$ over $ Z $ and g-pairs $(X/Z,B+M)$ and $(X'/Z,B'+M')$ such that $(X,B+M)$ is lc and the divisors $ M $ and $ M' $ are pushforwards of the same nef $ \R $-Cartier $\R$-divisor on a common birational model of $X$ and $X'$.
		
		\begin{center}
			\begin{tikzcd}
				(X,B+M) \arrow[rr, dashed, "\varphi"] \arrow[dr] && (X',B'+M') \arrow[dl] \\
				& Z
			\end{tikzcd}
		\end{center}
		
		\begin{enumerate}[(a)]
			\item The map $\varphi$ is called a \emph{minimal model in the sense of Birkar-Shokurov over $Z$} of the g-pair $(X,B+M)$ if 
			\begin{itemize}
				\item $ B' =\varphi_*B+E$, where $E$ is the sum of all $ \varphi^{-1} $-exceptional prime divisors on $ X' $,
				
				\item $X'$ is $\Q$-factorial,
				
				\item $K_{X'}+B'+M'$ is nef over $Z$, and 
				
				\item for any $\varphi $-exceptional prime divisor $ F $ on $ X $ we have 
				\[ a(F,X,B+M) < a(F,X',B'+M') . \]
			\end{itemize}
			
			If, moreover, $\varphi$ is a birational contraction, and $X'$ is not necessarily $\Q$-factorial if $X$ is not $\Q$-factorial (but $X'$ is required to be $\Q$-factorial if $X$ is $\Q$-factorial), then the map $\varphi$ is called a \emph{minimal model (in the usual sense) of $(X,B+M)$ over $ Z $}.
			
			Finally, a minimal model $ \varphi \colon (X,B+M) \dashrightarrow (X',B'+M') $ in the usual sense or in the sense of Birkar-Shokurov of $ (X,B+M) $ over $ Z $ is called \emph{good} if the divisor $ K_{X'} + B' + M' $ is semi-ample over $ Z $.
			
			\item The map $\varphi$ is called a \emph{Mori fiber space in the sense of Birkar-Shokurov over $Z$} of the g-pair $(X,B+M)$ if 
			\begin{itemize}
				\item $ B' =\varphi_*B+E$, where $E$ is the sum of all $ \varphi^{-1} $-exceptional prime divisors on $ X' $, 
				
				\item $X'$ is $\Q$-factorial,
				
				\item there exists a $ (K_{X'} + B' + M') $-negative extremal contraction $ X' \to T $ over $ Z $ with $ \dim X' > \dim T $, and
				
				\item for any divisorial valuation $ F $ over $ X $ we have
				\[ a(F,X,B+M) \leq a(F,X',B'+M') \]
				and the strict inequality holds if $ c_X(F) $ is a $\varphi $-exceptional prime divisor on $ X $.
			\end{itemize}
			
			If, moreover, $\varphi$ is a birational contraction, and $X'$ is not necessarily $\Q$-factorial if $X$ is not $\Q$-factorial (but $X'$ is required to be $\Q$-factorial if $X$ is $\Q$-factorial), then the map $ \varphi $ is called a \emph{Mori fiber space (in the usual sense) of $ (X,B+M) $ over $ Z $}.
		\end{enumerate}
	\end{dfn}
	
	We emphasize that in Definition \ref{dfn:MM_MFS} we allow a minimal model (resp.\ Mori fiber space) in the sense of Birkar-Shokurov to be lc and not only dlt; see \cite[Remark 2.4]{Hash18a} and \cite[p.\ 34, Comment]{Tsak21} for the justification. Furthermore, the g-pair $ (X',B'+M')$ in Definition \ref{dfn:MM_MFS} is lc. Indeed, if it is a (good) minimal model (in any sense) of $ (X,B+M) $, then this follows immediately from \cite[Lemma 2.8(i)]{LMT23}, while if it is a Mori fiber space (in any sense) of $ (X,B+M) $, then this follows from the above inequalities on discrepancies.
	
	\medskip
	
	We briefly discuss now the differences between the aforementioned two notions of a minimal model of a given g-pair. It is easy to check that minimal models in the usual sense and in the sense of Birkar-Shokurov coincide (modulo $\Q$-factoriality) in the klt case; see \cite[Remark 2.4(iii)]{Bir12b} and \cite[Subsection 2.2.4]{Tsak21}. The following result, which is an immediate consequence of Proposition \ref{prop:mainprop}, allows us to compare these two notions in the lc case as well, cf.\ \cite[Lemma 2.9(ii)]{LT22a}, \cite[Theorem 1.2(a)]{LT22b}.
	
	\begin{thm}
		Let $ (X/Z,B+M) $ be an NQC lc g-pair. Then $ (X,B+M) $ has a minimal model over $ Z $ if and only if $ (X,B+M) $ has a minimal model in the sense of Birkar-Shokurov over $ Z $.
	\end{thm}
	
	\begin{proof}
		If $ (X',B'+M') $ is a minimal model of $ (X,B+M) $ over $ Z $, then a dlt blow-up of $ (X',B'+M') $ is a minimal model in the sense of Birkar-Shokurov of $ (X,B+M) $ over $ Z $. The converse follows immediately from Proposition \ref{prop:mainprop}. Note also that the statement for $M=0$ follows from \cite[Theorem 1.7]{HH20}.
	\end{proof}
	
	The next remark is another immediate corollary of Proposition \ref{prop:mainprop}. It is very useful when one tries to construct minimal models of NQC lc g-pairs, as it allows one to work with NQC $\Q$-factorial dlt g-pairs instead, and thus it plays a key role in the proofs of Theorems \ref{thm:mAd_lc} and \ref{thm:Gon11_Thm1.1} and Corollaries \ref{cor:Hash20d_3.13} and \ref{cor:Hash20d_3.10}.
	
	\begin{rem} \label{rem:EMM_reduction}
		Let $ (X/Z,B+M) $ be an NQC lc g-pair. Let $ h \colon (T,B_T+M_T) \to (X,B+M) $ be a dlt blow-up of $ (X,B+M) $. If $ (Y,B_Y+M_Y) $ is a minimal model (in any sense) of $ (T,B_T+M_T) $ over $ Z $, then one can readily check that $ (Y,B_Y+M_Y) $ is a minimal model in the sense of Birkar-Shokurov of $ (X,B+M) $ over $ Z $, and therefore $(X,B+M)$ has a minimal model over $Z$ by Proposition \ref{prop:mainprop}.
	\end{rem}
	
	For the sake of completeness we also mention here that any two minimal models in the usual sense (resp.\ in the sense of Birkar-Shokurov) of an lc g-pair are isomorphic in codimension $ 1 $; see \cite[Lemma 2.14]{CT23} (resp.\ \cite[Lemma 3.1]{Chaud22c}).
	
	\begin{rem}\label{rem:WCM_EMM}
		With the same notation as in Definition \ref{dfn:MM_MFS}, if in part (a) we omit the second bullet and we replace the fourth bullet with the weaker condition \enquote{for any $\varphi $-exceptional prime divisor $ F $ on $ X $ we have $ a(F,X,B+M) \leq a(F,X',B'+M') $}, then we say that $(X',B'+M')$ is a \emph{weak canonical model in the sense of Birkar-Shokurov of $(X,B+M)$ over $Z$}; see \cite[Definition 2.26]{Tsak21} or \cite[Definition 3.2(2)]{HL23}. According to \cite[Proposition 2.33]{Tsak21} or \cite[Lemma 3.8]{HL23}, if an NQC lc g-pair $ (X/Z,B+M) $ has a weak canonical model in the sense of Birkar-Shokurov over $Z$, then it has a minimal model in the sense of Birkar-Shokurov over $Z$.
	\end{rem}
	
	The following basic result about good minimal models will be needed for the proof of Theorems \ref{thm:MM_NQC_NZD} and \ref{thm:MM_logbigmoduli}.
	
	\begin{lem}\label{lem:GMM}
		Let $(X/Z,B+M)$ be an lc g-pair. If $ (X,B+M) $ has a good minimal model in the usual sense or in the sense of Birkar-Shokurov over $ Z $, then every minimal model in the usual sense or in the sense of Birkar-Shokurov of $ (X,B+M) $ over $ Z $ is also good.
	\end{lem}
	
	\begin{proof}		
		If $ (X,B+M) $ has a good minimal model $ (V,B_V+M_V) $ over $ Z $, then a dlt blow-up of $ (V,B_V+M_V) $ is a good minimal model in the sense of Birkar-Shokurov of $ (X,B+M) $ over $ Z $, so it suffices to prove the statement under this assumption. Hence, assume now that there exists a good minimal model $ (X',B'+M') $ in the sense of Birkar-Shokurov of $ (X,B+M) $ over $ Z $. The first part of the proof below is similar to the proof of \cite[Remark 2.7]{Bir12a}, but we provide all the details for the convenience of the reader.
		
		Fix a minimal model $ (X'',B'' + M'') $ in the sense of Birkar-Shokurov of $ (X,B+M) $ over $ Z $, pick a (sufficiently high) common resolution of indeterminacies $ (p,q,r) \colon W \to X \times X' \times X'' $ of the maps $ X \dashrightarrow X' $ and $ X \dashrightarrow X'' $, 
		\begin{center}
			\begin{tikzcd}
				&&& X' \\
				W \arrow[rr, "p"] \arrow[urrr, bend left = 20, "q"] \arrow[drrr, bend right = 20, "r" swap] && X \arrow[ur, dashed] \arrow[dr, dashed] \\
				&&& X''
			\end{tikzcd}
		\end{center}
		and set 
		\[ E' := p^* (K_X+B+M) - q^* (K_{X'} + B' + M') \]
		and
		\[ E'' := p^* (K_X+B+M) - r^* (K_{X''} + B'' + M'') . \]
		By \cite[Remark 2.6]{LT22b} (see also \cite[Remark 2.6]{Bir12a}), $ E' $ is effective and $ q $-exceptional, while $ E'' $ is effective and $ r $-exceptional. Since $ q_* (E'' - E') \geq 0 $ and $ -(E'' - E') $ is $ q $-nef, by the Negativity lemma we deduce that $ E'' - E' \geq 0 $. Similarly, we have $ E' - E'' \geq 0 $. Therefore, $ E' = E'' $, which yields
		\[ q^* (K_{X'} + B' + M') = r^* (K_{X''} + B'' + M'') . \]
		Since $ K_{X'} + B' + M' $ is semi-ample over $ Z $ by assumption, we infer that $ K_{X''} + B'' + M'' $ is also semi-ample over $ Z $; in other words, $ (X'',B'' + M'') $ is a good minimal model in the sense of Birkar-Shokurov of $ (X,B+M) $ over $ Z $.\footnote{Note that this part of the proof is also valid if we assume instead that both $ (X',B'+M') $ and $ (X'',B''+M'') $ are only weak canonical models in the sense of Birkar-Shokurov over $ (X,B+M) $ over $ Z $; see also \cite[Lemma 3.5]{HL23}.}
		
		Finally, fix a minimal model $ (Y,B_Y+M_Y) $ of $ (X,B+M) $ over $ Z $ and consider a dlt blow-up
		$ h \colon (T,B_T+M_T) \to (Y,B_Y+M_Y) $
		of $ (Y,B_Y+M_Y) $. Then $ (T,B_T+M_T) $ is a minimal model in the sense of Birkar-Shokurov of $ (X,B+M) $ over $ Z $, which is actually a good minimal model in the sense of Birkar-Shokurov of $ (X,B+M) $ over $ Z $ by the previous paragraph; in particular, $ K_T + B_T + M_T $ is semi-ample over $ Z $. Since $ K_T + B_T + M_T \sim_\R h^*(K_Y + B_Y + M_Y) $, we infer that $ K_Y + B_Y + M_Y $ is also semi-ample over $ Z $; in other words, $ (Y,B_Y+M_Y) $ is a good minimal model of $ (X,B+M) $ over $ Z $.
	\end{proof}
	
	Next, we recall the definition of a canonical model of a g-pair and we briefly comment on this definition afterwards.
	
	\begin{dfn}
		Consider a diagram
		\begin{center}
			\begin{tikzcd}
				(X,B+M) \arrow[rr, dashed, "\varphi"] \arrow[dr] && (X',B'+M') \arrow[dl] \\
				& Z
			\end{tikzcd}
		\end{center}
		as in Definition \ref{dfn:MM_MFS}.
		If
		\begin{itemize}
			\item $ \varphi $ is a birational contraction,
			
			\item $ B' = \varphi_* B $,
			
			\item $K_{X'}+B'+M'$ is ample over $Z$, and 
			
			\item for any $\varphi $-exceptional prime divisor $ F $ on $ X $ we have 
			\[ a(F,X,B+M) \leq a(F,X',B'+M') , \]
		\end{itemize}
		then the map $ \varphi $ is called a \emph{canonical model of $ (X,B+M) $ over $ Z $}.
	\end{dfn}
	
	Note that the g-pair $ (X',B'+M')$ is lc by \cite[Lemma 2.8(i)]{LMT23}, and it is unique up to isomorphism by \cite[Lemma 2.12]{LMT23}.
	
	\medskip
	
	Finally, as promised at the end of the introduction, we show that klt generalized pairs of general type have good minimal models using the main result of \cite{BCHM10}, cf.\ \cite[Lemma 4.4(2)]{BZ16}.
	
	\begin{thm}\label{thm:EMM_general_type}
		Let $(X/Z,B+M)$ be a klt g-pair. The following statements hold:
		\begin{enumerate}[\normalfont (i)]
			\item If $K_X+B+M$ is pseudo-effective over $Z$ and if $B$ is big over $Z$, then $(X,B+M)$ has a good minimal model over $Z$.
			
			\item If $K_X+B+M$ is big over $Z$, then $(X,B+M)$ has a good minimal model over $Z$ as well as a canonical model over $Z$.
		\end{enumerate}
	\end{thm}
	
	\begin{proof}~
		
		\medskip
		
		\noindent (i) This is \cite[Lemma 4.2(ii)]{HanLiu20}.
		
		\medskip
		
		\noindent (ii)
		Since $K_X+B+M$ is big over $Z$, there exist an effective $\R$-Cartier $\R$-divisor $A$ on $X$ which is ample over $Z$ and an effective $\R$-Cartier $\R$-divisor $E$ on $X$ such that $A+E\sim_{\R,Z}K_X+B+M$. Since $(X,B+M)$ is klt, for any $ 0 < \varepsilon \ll 1 $ the g-pair $(X,B+\varepsilon A+\varepsilon E+M)$ with boundary part $B+\varepsilon A+\varepsilon E $ is also klt according to \cite[Remark 4.2(2)]{BZ16}. If we regard instead $B+\varepsilon E$ as the boundary part and $\varepsilon A + M$ as the nef part of the aforementioned g-pair, then by \cite[Lemma 3.4]{LX23b} there exists an effective $\R$-divisor $\Delta$ on $X$ such that $(X,\Delta)$ is a klt pair and 
		\[ K_X+\Delta \sim_{\R,Z} K_X+B+\varepsilon E+\varepsilon A+M \sim_{\R,Z} (1+\varepsilon) (K_X+B+M) . \]
		According to \cite{BCHM10}, the pair $(X,\Delta)$ has a good minimal model over $Z$ as well as a canonical model over $Z$, so the same holds for the g-pair $(X/Z,B+M)$, as asserted.
	\end{proof}
	
	\begin{rem}
		Theorem \ref{thm:EMM_general_type}(ii) can be proved alternatively as follows. Consider a \emph{small $\Q$-factorial modification} of $(X,B+M)$, namely, a $\Q$-factorial klt g-pair $(X',B'+M')$ with data $ W \overset{g}{\to} X' \to Z $ and $ M_W $, together with a small projective birational morphism $ h \colon X' \to X $ such that $K_{X'} + B' + M' \sim_\R h^* (K_X + B + M)$ and $ B' = h_*^{-1} B $; see \cite[Lemma 2.24(ii)]{Tsak21}. Since $ K_{X'} + B' +M' $ is big over $Z$, by \cite[Lemma 4.4(2)]{BZ16} we conclude that $(X',B'+M')$ has a good minimal model $(X'',B''+M'')$ over $Z$, and since $h$ is small, we can readily check now that $(X'',B''+M'')$ is also a good minimal model of $(X,B+M)$ over $Z$.
	\end{rem}
	
	\begin{cor}
		Let $(X,B+M)$ be a klt $\Q$-g-pair with data $ W \to X \overset{\pi}{\longrightarrow} Z $ and $ M_W $ such that either $K_X+B+M$ is pseudo-effective over $Z$ and $B$ is big over $Z$ or $K_X+B+M$ is big over $Z$. Then 
		\[ R(X/Z,K_X+B+M) := \bigoplus_{m \geq 0} \pi_* \OO_X \big( m(K_X+B+M) \big) \]
		is a finitely generated $ \OO_Z $-algebra.
	\end{cor}
	
	\begin{proof}
		Follows immediately from Theorem \ref{thm:EMM_general_type}.
	\end{proof}

	\subsection{The MMP for generalized pairs}
	\label{subsection:MMP_g-pairs}
	
	In this paper we use the foundations of the MMP for (not necessarily $ \Q $-factorial) NQC lc g-pairs, which were recently established in the papers \cite{HL23,LX23b,Xie22}. More precisely, by \cite[Theorem 1.1(1)-(4)]{HL23} and by \cite[Theorem 1.5]{Xie22} we now have a \emph{Cone theorem} and a \emph{Contraction theorem} for (not necessarily $ \Q $-factorial) NQC lc g-pairs, respectively, while \cite[Theorem 1.2]{LX23b} proved the \emph{existence of flips} in this setting. Therefore, given a (not necessarily $ \Q $-factorial) NQC lc g-pair $ (X/Z,B+M) $, we may run a $ (K_X + B + M) $-MMP over $Z$.
	
	\begin{rem}\label{rem:non-Q-fact_MMP}
		Let $ (X/Z,B+M) $ be an NQC lc g-pair and assume that $ K_X+B+M $ is not nef over $ Z $. By \cite[Theorem 1.1(1)]{HL23} there exists a $ (K_X+B+M) $-negative extremal ray $ R \in \NEb(X/Z) $ and by \cite[Theorem 1.5]{Xie22} we may consider the contraction $ g \colon X \to Y $ of $ R $. If $ \dim Y < \dim X $, then $ g $ is a \emph{Fano contraction}, that is, it determines a Mori fiber space structure on $(X/Z,B+M)$. If, on the other hand, $ \dim Y = \dim X $, then $ g $ is a birational contraction, and either $ \codim_X \Exc(g) = 1 $, in which case $g$ may contract more than one divisor, or $ \codim_X \Exc(g) \geq 2 $, in which case $g$ is small.
		
		In any of those two cases ($ \dim Y < \dim X $ or $ \dim Y = \dim X $), the numerical equivalence over $Y$ coincides with the $\R$-linear equivalence over $Y$. Indeed, let $D$ be an $\R$-Cartier divisor on $X$ such that $D \equiv_Y 0 $, or equivalently, $D \cdot R = 0$. Then $ D = \sum d_j D_j $, where each $d_j \in \R$ and each $D_j$ is a Cartier $\Z$-divisor on $X$ which is numerically trivial over $Y$; cf.\ \cite[Example 1.3.10]{Laz04}. In other words, $D_j \cdot R = 0$, and by (the third bullet of) \cite[Theorem 1.5]{Xie22} we deduce that $D_j \sim g^* G_j$ for some Cartier $\Z$-divisor $G_j$ on $Y$. Hence, $ D \sim_\R g^* \big( \sum d_j G_j \big)$, that is, $ D \sim_{\R,Y} 0 $, which proves the claim.
		
		Assume from now on that $g$ is a birational contraction. Then, as in the proof of \cite[Theorem 1.6]{Xie22}, we obtain a diagram
		\begin{center}
			\begin{tikzcd}[column sep = 3em, row sep = 3em]
				(X,B+M) \arrow[dr, "g" swap] \arrow[rr, dashed, "\varphi"] \arrow[ddr, bend right = 25pt] && (X',B'+M') \arrow[dl, "h"] \arrow[ddl, bend left = 25pt] \\
				& Y \arrow[d] \\
				& Z
			\end{tikzcd}
		\end{center}
		where the NQC lc g-pair $ (X'/Z,B'+M') $ is the canonical model of $ (X,B+M) $ over $ Y $; see also \cite[Section 4.9]{Fuj17book}. We emphasize that both $\varphi$ and $h$ are birational contractions by construction; see, for example, the proof of \cite[Theorem 1.2]{LX23b} for the details. 
		
		We also make the following observations, which will be useful later in the paper.
		\begin{enumerate}[(1)]            
			\item If $g$ contracts a prime divisor $F$ on $X$, then $\varphi$ contracts $F$ as well, since it cannot be an isomorphism at the generic point of $F$ according to \cite[Lemma 2.8(iii)(a)]{LMT23}.
			
			\item The birational contraction $h$ is small, regardless of whether $g$ is small or not. Indeed, arguing by contradiction and using the fact that $\varphi$ is a birational contraction, this follows readily from \cite[Lemma 2.8(iii)(b)]{LMT23}.
			
			\item It follows from Remark \ref{rem:WCM_EMM} that $ (X,B+M) $ has a minimal model in the sense of Birkar-Shokurov over $ Y $.
		\end{enumerate}
	\end{rem}
	
	We establish now some basic properties of the MMP for (not necessarily $\Q$-factorial) NQC lc g-pairs.
	
	\begin{lem}\label{lem:properties_of_MMP}
		Let $(X/Z,B+M)$ be an NQC lc g-pair. Consider a step of a $(K_X+B+M)$-MMP over $ Z $:
		\begin{center}
			\begin{tikzcd}
				(X,B+M) \arrow[rr, dashed, "\varphi"] \arrow[dr, "g" swap] && (X',B'+M') \arrow[dl, "h"] \\
				& Y
			\end{tikzcd}
		\end{center}
		Denote by $\K$ the field $\Q$ of rational numbers or the field $\R$ of real numbers. The following statements hold:
		\begin{enumerate}[\normalfont (i)]
			\item If $D$ is a $\K$-Cartier divisor on $X$, then $\varphi_*D$ is a $\K$-Cartier divisor on $X'$.
			
			\item The birational contraction $\varphi$ induces a linear map 
			\[ N^1(X/Z)_\K \to N^1(X'/Z)_\K , \ [D]_Z \mapsto [\varphi_* D]_Z . \]
			If, moreover, $\varphi$ is small, then the induced linear map is injective.
		\end{enumerate}
	\end{lem}
	
	\begin{proof}~
		
		\medskip
		
		\noindent (i) It suffices to treat the case $ \K = \Q $, so let $D$ be a $\Q$-Cartier divisor on $X$. We first claim that the map $\varphi$ is also a step of a $(K_X+\hat{B}+\hat{M})$-MMP over $Z$ for some lc $\Q$-g-pair $(X,\hat{B}+\hat{M})$. To prove this assertion, by \cite[Theorem 1.4]{Chen23} we may find positive real numbers $r_1,\dots,r_\ell$ and $\Q$-divisors $B_1,\dots, B_\ell$ and $M_1,\dots, M_\ell$ on $X$ such that 
		$$\sum_{j=1}^\ell r_j=1, \quad B = \sum_{j=1}^\ell r_j B_j,\quad M = \sum_{j=1}^\ell r_j M_j, $$
		each $\Q$-g-pair $(X,B_j+M_j)$ is lc, each divisor $ K_X + B_j + M_j $ is $ \Q $-Cartier, and we have
		\begin{equation}\label{eq:1_properties_of_MMP}
			K_X + B + M = \sum_{j=1}^\ell r_j \big( K_X + B_j + M_j \big) . 
		\end{equation}
		Since $\rho(X/Y) = 1$, for each $1 \leq j \leq \ell$ there exists $\alpha_j \in \R$ such that
		\begin{equation}\label{eq:2_properties_of_MMP}
			K_X+B+M \equiv_Y \alpha_j (K_X+B_j+M_j) .
		\end{equation}
		Since $-(K_X+B+M)$ is ample over $Y$, it holds that $\alpha_j \neq 0 $ for every $1 \leq j \leq \ell$. Moreover, at least one of the $\alpha_j$, say $\alpha_1$, must be positive, since otherwise each divisor $K_X+B_j+M_j$ would be ample over $Y$ by \eqref{eq:2_properties_of_MMP}, and hence $ K_X+B+M $ would also be ample over $Y$ by \eqref{eq:1_properties_of_MMP}, which is impossible. Therefore, $(X/Z,B_1+M_1)$ is an lc $\Q$-g-pair such that $-(K_X+B_1+M_1)$ is ample over $Y$. Then, by definition of an MMP step (see also the proof of \cite[Theorem 1.6]{Xie22} and \cite[Section 4.9]{Fuj17book}), we have
		\[ X' \simeq \proj_Y \bigg( \bigoplus_{m \geq 0} g_*  \OO_X \big( m (K_X +B_1 + M_1) \big) \bigg) . \]
		Thus, $\varphi$ is a step of a $(K_X+B_1+M_1)$-MMP over $Z$.
		
		Consequently, we may assume that $(X,B+M)$ itself is a $\Q$-g-pair. We denote by $R$ the $(K_X+B+M)$-negative extremal ray contracted by $g = \operatorname{cont}_R$. Since $\rho(X/Y) = 1$, there exists $\mu \in \Q$ such that 
		$$ \big( D+\mu(K_X+B+M) \big) \cdot R = 0. $$
		Take a positive integer $m$ such that $ m \big( D+\mu(K_X+B+M) \big) $ is Cartier. By (the third bullet of) \cite[Theorem 1.5]{Xie22} there exists a Cartier $\Z$-divisor $G$ on $Y$ such that 
		$$ m \big( D+\mu(K_X+B+M) \big) \sim g^* G . $$
		Therefore, 
		$$ m \big( \varphi_*D + \mu(K_{X'}+B'+M') \big) \sim h^* G $$ 
		is a Cartier $\Z$-divisor on $X'$; see also Remark \ref{rem:non-Q-fact_MMP}(2). Since $K_{X'}+B'+M'$ is itself $\Q$-Cartier by construction, we conclude that $\varphi_* D$ is a $\Q$-Cartier divisor on $X'$, which completes the proof of (i).
		
		\medskip
		
		\noindent (ii) The existence of the linear map 
		\[ N^1(X/Z)_\K \to N^1(X'/Z)_\K , \ [D]_Z \mapsto [\varphi_* D]_Z \]
		follows immediately from (i).
		
		Assume now that $\varphi$ is small and also that $D' := \varphi_* D \equiv_Z 0$. We will show that $D \equiv_Z 0 $. To this end, consider a resolution of indeterminacies $ (p,q) \colon W \to X \times X' $ of the map $\varphi$.
		\begin{center}
			\begin{tikzcd}[column sep = 3em, row sep = 2.5em]
				& W \arrow[dl, "p" swap] \arrow[dr, "q"] \\
				X \arrow[rr, dashed, "\varphi"] \arrow[dr, "g" swap] && X' \arrow[dl, "h"] \\
				& Y
			\end{tikzcd}
		\end{center}
		Since $D'$ is nef over $Y$, we can readily check that $-(p^*D-q^*D')$ is nef over $X$. Since $\varphi$ is small, the $\K$-Cartier $\K$-divisor $p^*D-q^*D'$ is $p$-exceptional, and it follows now from the Negativity lemma that $p^*D-q^*D'\geq 0$. Since $-D'$ is also nef over $Y$, we have $q^*D'-p^*D\ge0$ as well, and hence $p^*D=q^*D'$. It follows that $D \equiv_Z 0 $, which completes the proof of (ii).
	\end{proof}
	
	Even though we may run a $ (K_X + B + M) $-MMP over $Z$ for any given NQC lc g-pair $ (X/Z,B+M) $, its termination is not known in general. However, \cite{CT23} establishes the termination of flips (and hence of any MMP) for all NQC lc g-pairs of dimension $3$ as well as for NQC lc g-pairs of dimension $4$ whose canonical class is pseudo-effective; see also \cite{HM20,HanLiu22}.
	
	As the next result demonstrates, we may also run MMPs with scaling in this very general setting. Their termination is also unsettled in general, but there are already several results in the literature regarding this termination problem. We refer to \cite{BZ16,LT22a,LT22b,LX23a} and the relevant references therein for more information. See also Theorem \ref{thm:LT22b_4.1} below, which constitutes an exact analogue of \cite[Theorem 4.1(iii)]{Bir12a} in the setting of g-pairs, as well as Sections \ref{section:applications_I} and \ref{section:applications_II} for further developments.
	
	\begin{lem}\label{lem:MMP_with_scaling}
		Let $ (X/Z,B+M) $ be an NQC lc g-pair. Let $ P $ be the pushforward to $ X $ of an NQC divisor (over $ Z $) on a birational model of $ X $ and let $ N $ be an effective $ \R $-divisor on $ X $ such that $ N+P $ is $ \R $-Cartier. Assume that the NQC g-pair $ \big( X, (B+N) + (M+P) \big) $ is lc and that the divisor $ K_X + B + N + M + P $ is nef over $ Z $. Then we may run a $ (K_X+B+M) $-MMP over $ Z $ with scaling of $ N+P $. 
		
		In particular, we may run a $ (K_X + B + M) $-MMP over $Z$ with scaling of an ample divisor.
	\end{lem}
	
	\begin{proof}
		Due to \cite[Theorem 1.1(1)(2)]{HL23} and \cite[Proposition 2.6]{HanLiu20}, it is easy to check that \cite[Lemma 3.23]{HanLi22} holds if the assumption that $ X $ is $ \Q $-factorial klt (which is present in op.\ cit.\@) is replaced by the assumption that $ N+P $ is $ \R $-Cartier (which is included in the above statement). Thus, taking the second paragraph of Remark \ref{rem:non-Q-fact_MMP} into account (for the repetitions of the procedure), the previous observation and \cite[Theorem 1.5]{Xie22} imply that one may run a $ (K_X+B+M) $-MMP over $ Z $ with scaling of $ N+P $. Now, regarding the last sentence of the lemma, set $ N = 0 $ and take $ P $ to be a general ample over $ Z $ $\R$-divisor on $X$ to conclude.
	\end{proof}
	
	The recently established Contraction theorem for (not necessarily $\Q$-factorial) NQC lc g-pairs has even further and significant consequences. Specifically, it allows us to remove the $ \Q $-factoriality assumption from the majority of the results of the paper \cite{LT22b}. We indicate now such refinements of \cite[Lemma 2.13, Theorem 2.14, Lemma 2.16 and Theorem 4.1]{LT22b}. This discussion, however, will be completed later in the paper, where \cite[Proposition A.3, Theorem 1.2, Corollaries 1.3 and 1.4, and Theorem 1.5]{LT22b} will also be refined accordingly.
	
	\begin{rem}\label{rem:comments_on_LT22b}
		For ease of reference we will use here the same notation as the one from those parts of \cite{LT22b} which will be mentioned below.
		\begin{enumerate}[(1)]
			\item Taking Lemmas \ref{lem:pullback_P+N} and \ref{lem:MMP_with_scaling} into account, it is straightforward to check that 
			\cite[Lemma 2.13 and Theorem 2.14]{LT22b} also hold without the assumption that the underlying variety $ X_1 $ of the given NQC lc g-pair $ \big( X_1/Z , (B_1+N_1) + (M_1+P_1) \big) $ is $ \Q $-factorial; see also Section \ref{section:lifting_MMP} for further developments.
			
			\item In view of Lemma \ref{lem:MMP_with_scaling} and \cite[Section 3]{LX23a}, one may readily check that \cite[Lemma 2.16]{LT22b} also holds without the assumption that the underlying variety $ X $ of the given NQC lc g-pair $ \big( X/Z , (B+N) + (M+P) \big) $ is $ \Q $-factorial.
		\end{enumerate}
	\end{rem}
	
	\begin{thm}\label{thm:LT22b_4.1}
		Let $ (X/Z,B+M) $ be an NQC lc g-pair. Let $ P $ be the pushforward to $X$ of an NQC divisor (over $ Z $) on a high birational model of $ X $ and let $ N $ be an effective $ \R $-divisor on $ X $ such that $ N+P $ is $ \R $-Cartier. Assume that the NQC g-pair $ \big( X, (B+N) + (M+P) \big) $ is lc and that the divisor $ K_X + B + N + M + P $ is nef over $ Z $. Consider a $ (K_X + B + M) $-MMP over $Z$ with scaling of $ N+P $, denote by $ \lambda_i $ the corresponding nef thresholds at the steps of this MMP and set $ \lambda := \lim\limits_{i\to +\infty} \lambda_i $.
		
		If $ \lambda \neq \lambda_i $ for every $ i $ and if $\big(X,(B+\lambda N)+(M+\lambda P)\big)$ has a minimal model in the sense of Birkar-Shokurov over $Z$, then the given MMP terminates.
	\end{thm}
	
	\begin{proof}
		Taking Lemma \ref{lem:MMP_with_scaling} into account, we may repeat verbatim the proof of \cite[Theorem 4.1]{LT22b}, except that we now invoke the refined version of \cite[Theorem 2.14]{LT22b}, which was discussed in Remark \ref{rem:comments_on_LT22b}(1), in Step 5 of that proof\footnote{For the sake of clarity we note that \cite[Lemma 2.16]{LT22b} was (only) applied in Step 5 of the proof of \cite[Theorem 4.1]{LT22b} to NQC $ \Q $-factorial dlt g-pairs. Therefore, the slightly more general version of \cite[Lemma 2.16]{LT22b} mentioned in Remark \ref{rem:comments_on_LT22b}(2) is not required for the proof of Theorem \ref{thm:LT22b_4.1}.}.
	\end{proof}
	
	The final result in this subsection exploits the \emph{boundedness of the length of extremal rays}, namely \cite[Theorem 1.1(2)]{HL23}, and plays a fundamental role in the proofs of Theorem \ref{thm:mainthm} and Lemma \ref{lem:scaling_numbers_to_zero}. It is a variant of \cite[Lemma 3.21]{HanLi22} and constitutes a generalization of \cite[Lemma 2.20]{LT22a}, cf.\ \cite[Proposition 2.12]{LT22b}. For brevity we only indicate below the necessary modifications to the proof of \cite[Lemma 2.20]{LT22a} and we refer to op.\ cit.\ for the details. We stress that the proof of Lemma \ref{lem:trivial_MMP} also relies essentially on the Contraction theorem for NQC lc g-pairs, namely \cite[Theorem 1.5]{Xie22}.
	
	\begin{lem}\label{lem:trivial_MMP}
		Let $ (X/Z,B+M) $ be an NQC lc g-pair. Let $ P $ be the pushforward to $ X $ of an NQC divisor (over $ Z $) on a birational model of $ X $ and let $ N $ be an effective $ \R $-divisor on $ X $ such that $ N+P $ is $ \R $-Cartier. Assume that the NQC g-pair $ \big( X, (B+N) + (M+P) \big) $ is lc and that the divisor $ K_X + B + N + M + P $ is nef over $ Z $. Then there exists $ \varepsilon_0 >0 $ such that for every $ \varepsilon \in (0,\varepsilon_0)$, any $ \big(K_X + B + M + (1-\varepsilon)(N+P) \big) $-MMP over $Z$ is $ (K_X + B + N + M + P) $-trivial.
	\end{lem}
	
	\begin{proof}
		To prove the statement, we argue as in the proof of \cite[Lemma 2.20]{LT22a}, the difference being that we now work with the g-pair $ \big( X, (B+N) + (M+P) \big) $ and we invoke \cite[Proposition 2.6]{HanLiu20}, (the third bullet of) \cite[Theorem 1.5]{Xie22} and \cite[Theorem 1.1(2)]{HL23} instead of [HL18, Prop.\ 3.16] = \cite[Proposition 3.20]{HanLi22}, \cite[Theorem 3.25(4)]{KM98} and [HL18, Prop.\ 3.13] = \cite[Proposition 3.17]{HanLi22}, respectively.
	\end{proof}

	\subsection{Log abundant generalized pairs}
	\label{subsection:log_abundant}
	
	Throughout this subsection we work exclusively in the absolute setting, that is, we assume that $ Z = \spec \C $. Therefore, $ X $ always denotes here a normal \emph{projective} variety.
	
	Given a normal projective variety $ X $ and an $ \R $-Cartier $ \R $-divisor $ D $ on $ X $, we denote by $ \kappa_\iota(X,D) $ the \emph{invariant Iitaka dimension} of $ D $ and by $ \kappa_\sigma(X,D) $ the \emph{numerical dimension} of $ D $; see \cite{Choi08} and \cite{Nak04}, respectively.
	We say that $ D $ is \emph{abundant} if the equality $ \kappa_\iota(X,D) = \kappa_\sigma(X,D) $ holds. In particular, we say that an lc g-pair $ (X,B+M) $ is \emph{abundant} if the divisor $ K_X+B+M $ is abundant.
	
	\begin{rem}\label{rem:abundant_MM}
		Let $(X,B+M)$ be an lc g-pair and let $ \varphi \colon (X,B+M) \dashrightarrow (X',B'+M')$ be a minimal model (in any sense) of $(X,B+M)$. Then $(X,B+M)$ is abundant if and only if $(X',B'+M')$ is abundant. Indeed, for any resolution of indeterminacies $ (p,q) \colon W \to X \times X' $ of the map $\varphi$ we may write
		\[ p^* (K_X+B+M) \sim_\R q^* (K_{X'}+B'+M') + E , \]
		where $E$ is an effective $q$-exceptional $\R$-Cartier $\R$-divisor on $W$ (see \cite[Lemma 2.8(i)]{LMT23} and \cite[Remark 2.6]{LT22b}, respectively), and it follows now from \cite[Remark 2.15(2)]{Hash20d} that $K_X+B+M$ is abundant if and only if $K_{X'}+B'+M'$ is abundant.
	\end{rem}
	
	The following definition will be needed in Section \ref{section:applications_II}. For more information about the notion that will be defined below as well as for its relative version we refer to \cite[Subsection 2.3]{Hash20d}.
	
	\begin{dfn}
		Let $ (X,B+M) $ be an lc g-pair, where $ X $ is a normal projective variety. An $ \R $-Cartier $ \R $-divisor $ D $ on $ X $ is said to be \emph{log abundant with respect to $ (X,B+M) $} if $ D $ is abundant and for any lc center $ S $ of $ (X,B+M) $ with normalization $ \nu \colon S^\nu \to S $ the divisor $ \nu^* D $ is abundant. In particular, we say that the given lc g-pair $ (X,B+M) $ is \emph{log abundant} if the divisor $ K_X+B+M $ is log abundant with respect to $ (X,B+M) $.
	\end{dfn}
	
	\begin{lem}\label{lem:dlt_blow-up_log_abundant}
		Let $(X,B+M)$ be an lc g-pair and let $h \colon (X',B'+M') \to (X,B+M) $ be a dlt blow-up of $(X,B+M)$. Then $(X,B+M)$ is log abundant if and only if $(X',B'+M')$ is log abundant.
	\end{lem}
	
	\begin{proof}
		Since
		$$ K_{X'} + B' + M' \sim_{\R} h^* (K_X + B + M) , $$
		by \cite[Remark 2.15(2)]{Hash20d} we deduce that $K_X+B+M$ is abundant if and only if $K_{X'}+B'+M'$ is abundant.
		
		For any divisorial valuation $E$ over $X$ we have $a(E,X,B+M) = a(E,X',B'+M')$, and thus $c_X(E)$ is an lc center of $(X,B+M)$ if and only if $c_{X'}(E)$ is an lc center of $(X',B'+M')$. Therefore, given an lc center $W'$ of $(X',B'+M')$, its image $W := h(W')$ is an lc center of $(X,B+M)$. Since $W'$ is normal (see \cite[Subsection 2.3]{HanLi22}), the restriction $ h|_{W'} \colon W' \to W $ of $h$ to $W'$ factors through the normalization $W^\nu$ of $W$, so we obtain a projective surjective morphism $ \xi \colon W' \to W^\nu $. Since 
		\[ (K_{X'}+B'+M')|_{W'} \sim_\R \xi^* \big( (K_X+B+M)|_{W^\nu} \big) , \]
		by \cite[Remark 2.15(2)]{Hash20d} we deduce that $(K_X+B+M)|_{W^\nu}$ is abundant if and only if $(K_{X'}+B'+M')|_{W'}$ is abundant, which proves the statement.
	\end{proof}
	
	\begin{lem}\label{lem:log_abundant_MMs}
		Let $(X,B+M)$ be an lc g-pair and let $(X',B'+M')$ and $(X'',B''+M'')$ be two minimal models (in any sense) of $(X,B+M)$. Then $(X',B'+M')$ is log abundant if and only if $(X'',B''+M'')$ is log abundant.
	\end{lem}
	
	\begin{proof}
		In view of Lemma \ref{lem:dlt_blow-up_log_abundant}, it suffices to treat the case when both $(X',B'+M')$ and $(X'',B''+M'')$ are minimal models in the sense of Birkar-Shokurov of $(X,B+M)$ and have dlt singularities as well. Then the lc centers of $(X',B'+M')$ and $(X'',B''+M'')$ are normal (and finitely many); see \cite[Subsection 2.3]{HanLi22}.
		
		Pick a sufficiently high common resolution of indeterminacies $ (p,q,r) \colon W \to X \times X' \times X'' $ of the maps $ X \dashrightarrow X' $ and $ X \dashrightarrow X'' $ such that all lc centers of $(X',B'+M')$ and $(X'',B''+M'')$ have been extracted on $ W $. Then the occurring equality $E' = E''$ (see the second paragraph of the proof of Lemma \ref{lem:GMM}) implies
		\begin{equation}\label{eq:1_log_abundant_MMs}
			a(D,X',B'+M') = a(D,X'',B''+M'') \ \text{ for any prime divisor $D$ on $W$} 
		\end{equation}
		and
		\begin{equation}\label{eq:2_log_abundant_MMs}
			q^* (K_{X'} + B' + M') = r^* (K_{X''} + B'' + M'') .
		\end{equation}
		By \eqref{eq:1_log_abundant_MMs} and by construction, for any lc center $S$ of $(X',B'+M')$ there exists a prime divisor $D$ on $W$ whose image $T := r(D)$ is an lc center of $(X'',B''+M'')$, and vice versa. Then the statement follows readily by using \eqref{eq:2_log_abundant_MMs} and by invoking \cite[Remark 2.15(2)]{Hash20d}, bearing also the proof of Lemma \ref{lem:dlt_blow-up_log_abundant} in mind.
	\end{proof}

	\section{Lifting an MMP}
	\label{section:lifting_MMP}
	
	Our first goal in this section is to demonstrate how one can \enquote{lift} an MMP starting with an NQC lc g-pair to an MMP starting with an NQC $\Q$-factorial dlt g-pair, cf.\ \cite[Section 3]{LMT23}.
	
	\medskip
	
	Let $ ( X_1/Z, B_1+M_1) $ be an NQC lc g-pair. Assume that the divisor $K_{X_1} + B_1 + M_1$ is not nef over $Z$. Consider the first step
	\begin{center}
		\begin{tikzcd}[column sep = 2em, row sep = large]
			(X_1,B_1+M_1) \arrow[rr, dashed, "\pi_1"] \arrow[dr, "\theta_1" swap] && (X_2,B_2+M_2) \arrow[dl, "\theta_1^+"] \\
			& Z_1
		\end{tikzcd}
	\end{center}
	of a $ (K_{X_1}+B_1+M_1)$-MMP over $Z$, assuming (for our purposes here) that $\theta_1$ is not a Mori fiber space, and a dlt blow-up 
	$ h_1 \colon (X_1', B_1' + M_1') \to (X_1,B_1+M_1) $
	of $ (X_1,B_1+M_1) $. 
	By arguing as in the first paragraph of the proof of \cite[Lemma 2.13]{LT22b} and by taking Remark \ref{rem:non-Q-fact_MMP}(3) into account, we may construct the following diagram:
	\begin{center}
		\begin{tikzcd}[column sep = 2em, row sep = large]
			(X_1',B_1'+M_1') \arrow[d, "h_1" swap] \arrow[rr, dashed, "\rho_1"] && (X_2',B_2'+M_2') \arrow[d, "h_2" swap]
			\\ 
			(X_1,B_1+M_1) \arrow[dr, "\theta_1" swap] \arrow[rr, dashed, "\pi_1"] && (X_2,B_2+M_2) \arrow[dl, "\theta_1^+"] 
			\\
			& Z_1
		\end{tikzcd}
	\end{center}
	where the map $ \rho_1 \colon X_1' \dashrightarrow X_2' $ is a $(K_{X_1'}+B_1'+M_1')$-MMP over $ Z_1 $ (with scaling of an ample divisor) and the map $ h_2 \colon (X_2',B_2'+M_2') \to (X_2,B_2+M_2) $ is a dlt blow-up of $ (X_2,B_2+M_2) $.
	
	\begin{rem}\label{rem:lifting_Picard_number_drops}
		If $\theta_1$ contracts a prime divisor $D$ on $X_1$, then so does $\pi_1$ according to Remark \ref{rem:non-Q-fact_MMP}(1). Furthermore, the strict transform $ (h_1)_*^{-1} D $ of $D$ on $X_1'$ must be contracted by $\rho_1$. Indeed, by \cite[Lemma 2.8(iii)(a)]{LMT23} we obtain 
		\[ a(D,X_1',B_1'+M_1') = a(D,X_1,B_1+M_1) < a(D,X_2,B_2+M_2) = a(D,X_2',B_2'+M_2') , \]
		which shows that $\rho_1$ cannot be an isomorphism at the generic point of $ (h_1)_*^{-1} D $ and proves the previous claim. In particular, it follows from the above and from \cite[Corollary 5.10 and Theorem 6.1]{HL23} that $ \rho(X'_2/Z) < \rho(X'_1/Z) $.
	\end{rem}
	
	If we now have a $ (K_{X_1}+B_1+M_1)$-MMP over $Z$, then by repeating the above procedure we obtain the following result, which plays a crucial role in the proof of Proposition \ref{prop:mainprop}.
	
	\begin{thm}\label{thm:lifting_MMP}
		Let $ ( X_1/Z, B_1+M_1) $ be an NQC lc g-pair. Consider a $ (K_{X_1}+B_1+M_1)$-MMP over $Z$:
		\begin{center}
			\begin{tikzcd}[column sep = 1em, row sep = large]
				(X_1,B_1+M_1) \arrow[dr, "\theta_1" swap] \arrow[rr, dashed, "\pi_1"] && (X_2,B_2+M_2) \arrow[dl, "\theta_1^+"] \arrow[dr, "\theta_2" swap] \arrow[rr, dashed, "\pi_2"] && (X_3,B_3+M_3) \arrow[dl, "\theta_2^+"] \arrow[rr, dashed, "\pi_3"] && \cdots \\
				& Z_1 && Z_2
			\end{tikzcd}
		\end{center}
		Then there exists a diagram
		\begin{center}
			\begin{tikzcd}[column sep = 1em, row sep = large]
				(X_1',B_1'+M_1') \arrow[d, "h_1" swap] \arrow[rr, dashed, "\rho_1"] && (X_2',B_2'+M_2') \arrow[d, "h_2" swap] \arrow[rr, dashed, "\rho_2"] && (X_3',B_3'+M_3') \arrow[d, "h_3" swap] \arrow[rr, dashed, "\rho_3"] && \dots 
				\\ 
				(X_1,B_1+M_1) \arrow[dr, "\theta_1" swap] \arrow[rr, dashed, "\pi_1"] && (X_2,B_2+M_2) \arrow[dl, "\theta_1^+"] \arrow[dr, "\theta_2" swap] \arrow[rr, dashed, "\pi_2"] && (X_3,B_3+M_3) \arrow[dl, "\theta_2^+"] \arrow[rr, dashed, "\pi_3"] && \dots \\
				& Z_1 && Z_2
			\end{tikzcd}
		\end{center}
		where, for each $i \geq 1$, 
		\begin{itemize}
			\item the map $ \rho_i \colon X_i' \dashrightarrow X_{i+1}' $ is a $(K_{X_i'}+B_i'+M_i')$-MMP over $ Z_i $, and
			
			\item the map $ h_i \colon (X_i',B_i'+M_i') \to (X_i,B_i+M_i) $ is a dlt blow-up.
		\end{itemize}
		In particular, the sequence on top of the above diagram is a $ (K_{X_1'}+B_1'+M_1') $-MMP over $ Z $, where $ (X_1'/Z,B_1'+M_1') $ is an NQC $ \Q $-factorial dlt g-pair.
		
		Furthermore, there exists an integer $k \geq 1$ such that the map $ \pi_i \colon X_i \dashrightarrow X_{i+1} $ is small and the induced linear map $ N^1(X_i/Z)_\R \to N^1(X_{i+1}/Z)_\R $ is an isomorphism for each $ i \geq k $; in particular, it holds that $ \rho(X_i/Z) = \rho(X_{i+1}/Z) $ for each $ i \geq k $.
	\end{thm}
	
	\begin{proof}
		It remains to prove the second part of the statement.
		In view of Lemma \ref{lem:properties_of_MMP}(ii) and Remark \ref{rem:lifting_Picard_number_drops}, we may find an integer $ \ell_1 \geq 1 $ such that for each $ i \geq \ell_1 $ the birational contraction $ \pi_i \colon X_i \dashrightarrow X_{i+1} $ is small and the induced linear map $ N^1(X_i/Z)_\R \to N^1(X_{i+1}/Z)_\R $ is injective. By relabelling the given $ (K_{X_1}+B_1+M_1)$-MMP over $Z$, we may assume that $ \ell_1 = 1 $. We obtain thus a non-decreasing sequence $ \big\{ \rho(X_i/Z) \big\}_{i=1}^{+\infty} $ of positive integers which is bounded from above by $ \rho(X_1'/Z) < +\infty $; see \cite[Corollary 5.10 and Theorem 6.1]{HL23}. Therefore, this sequence must stabilize, that is, there exists an integer $ \ell_2 \geq 1 $ such that for each $ i \geq \ell_2 $ the linear map $ N^1(X_i/Z)_\R \to N^1(X_{i+1}/Z)_\R $ is an isomorphism. We are done by taking $ k := \ell_2 \geq \ell_1 \geq 1 $.
	\end{proof}
	
	Our second goal in this section is to explain how one can make an analogous construction when one considers an MMP \emph{with scaling} starting with an NQC lc g-pair, cf.\ \cite[Subsection 2.5]{LT22b}.
	
	\medskip
	
	
	Let the g-pair $ ( X_1/Z, B_1+M_1) $ and the divisors $P_1$ and $N_1$ be as in Lemma \ref{lem:MMP_with_scaling} so that we may run a $ (K_{X_1}+B_1+M_1)$-MMP over $Z$ with scaling of $ N_1 + P_1 $. Assume that the divisor $K_{X_1} + B_1 + M_1$ is not nef over $Z$, and set
	\[ \lambda_1 := \inf \{ t \in \R_{\geq 0} \mid K_{X_1} + (B_1 + t N_1) + (M_1 + t P_1) \text{ is nef over } Z \} \in (0,1] . \]
	Consider the first step
	\begin{center}
		\begin{tikzcd}[column sep = 2em, row sep = large]
			(X_1,B_1+M_1) \arrow[rr, dashed, "\pi_1"] \arrow[dr, "\theta_1" swap] && (X_2,B_2+M_2) \arrow[dl, "\theta_1^+"] \\
			& Z_1
		\end{tikzcd}
	\end{center}
	of a $ (K_{X_1}+B_1+M_1)$-MMP over $Z$ with scaling of $N_1+P_1$, assuming (for our purposes here) that $\theta_1$ is not a Mori fiber space, as well as a dlt blow-up 
	$ h_1 \colon (X_1', B_1' + M_1') \to (X_1,B_1+M_1) $
	of $ (X_1,B_1+M_1) $; in particular, we have 
	\[ K_{X_1'} + B_1'+M_1' \sim_\R h_1^*(K_{X_1} + B_1 + M_1) . \]
	Let $f_1\colon W\to X_1$ be a log resolution of $(X_1,B_1+N_1)$ which factors through $X_1'$ and such that there exists an $\R$-divisor $P_W$ on $W$ such that $P_W$ is NQC (over $Z$) and $ (f_1)_* P_W = P $. By Lemma \ref{lem:pullback_P+N} we may write 
	\[ f_1^*(P_1+N_1) = P_W + (f_1)_*^{-1} N_1 + E_1 , \]
	where $ E_1 $ is an effective $ f_1 $-exceptional $ \R $-divisor on $ W $. We now define $N_1'$ and $P_1'$ as the pushforwards of $(f_1)_*^{-1} N_1+E_1$ and $P_W$, respectively, to $X_1'$, and we note that
	\begin{equation}\label{eq:1_lifting_MMP_with_scaling}
		N_1' + P_1' = h_1^* (N_1 + P_1) .
	\end{equation}
	By arguing as in the proof of \cite[Lemma 2.13]{LT22b} and by taking Remark \ref{rem:non-Q-fact_MMP}(3) into account, we may construct the following diagram:
	\begin{center}
		\begin{tikzcd}[column sep = 2em, row sep = large]
			(X_1',B_1'+M_1') \arrow[d, "h_1" swap] \arrow[rr, dashed, "\rho_1"] && (X_2',B_2'+M_2') \arrow[d, "h_2" swap]
			\\ 
			(X_1,B_1+M_1) \arrow[dr, "\theta_1" swap] \arrow[rr, dashed, "\pi_1"] && (X_2,B_2+M_2) \arrow[dl, "\theta_1^+"] 
			\\
			& Z_1
		\end{tikzcd}
	\end{center}
	where the map $ \rho_1 \colon X_1' \dashrightarrow X_2' $ is a $(K_{X_1'}+B_1'+M_1')$-MMP over $ Z_1 $ with scaling of $ N_1' + P_1' $ and the map $ h_2 \colon (X_2',B_2'+M_2') \to (X_2,B_2+M_2) $ is a dlt blow-up of $ (X_2,B_2+M_2) $. Moreover, this MMP is also a $(K_{X_1'}+B_1'+M_1')$-MMP over $ Z $ with scaling of $  N_1' + P_1' $, and if we set $P_2 := (\pi_1)_* P_1 $, $N_2 := (\pi_1)_* N_1 $, $ P_2' :=(\rho_1)_* P_1' $ and $N_2' :=(\rho_1)_* N_1'$, then we have
	\begin{equation}\label{eq:2_lifting_MMP_with_scaling}
		N_2' + P_2' = h_2^* (N_2 + P_2) .
	\end{equation}
	
	A priori, the map $ \rho_1 \colon X_1' \dashrightarrow X_2' $ is a $(K_{X_1'}+B_1'+M_1')$-MMP over $ Z_1 $ with scaling of an ample divisor. However, the following crucial observation, which is also contained in the proof of \cite[Lemma 2.13]{LT22b}, allows us to view $\rho_1$ as a $(K_{X_1'}+B_1'+M_1')$-MMP with scaling of $ N_1' + P_1' $ over $Z_1$ or over $ Z $, as mentioned above. Specifically, if we denote by $ Y^j \dashrightarrow Y^{j+1} $ the steps of the $(K_{X_1'}+B_1'+M_1')$-MMP over $ Z_1 $ with scaling of an ample divisor, where $ Y^1 := X_1' $ and $ Y^k := X_2' $, and by $B^j$, $M^j$, $N^j$ and $P^j$ the pushforwards of $B_1'$, $M_1'$, $N_1'$ and $P_1'$, respectively, to $Y^j$, and if we consider the nef thresholds $\nu_j$ at the steps of this MMP, i.e., 
	$$ \nu_j := \inf \big\{ t \in \R_{\geq 0} \mid K_{Y^j} + (B^j + t N^j) + (M^j + t P^j) \text{ is nef over } Z \big\} , \ j \in \{1,\dots,k \} , $$
	then we can check that $ \nu_j = \lambda_1 $ for every $ j \in \{1, \dots, k-1 \} $.
	
	Therefore, if we are given instead a $ (K_{X_1}+B_1+M_1)$-MMP over $Z$ with scaling of $N_1 + P_1$, then by repeating the above procedure and by taking the previous observation, Lemma \ref{lem:properties_of_MMP}(ii) and Remark \ref{rem:lifting_Picard_number_drops} into account, we obtain the following result, which plays a central role in the proofs of Corollaries \ref{cor:Hash20d_3.9} and \ref{cor:Hash20d_3.13}.
	
	\begin{thm} \label{thm:lifting_MMP_with_scaling}
		Let $ (X_1/Z,B_1+M_1) $ be an NQC lc g-pair. Let $ P_1 $ be the pushforward to $ X_1 $ of an NQC divisor (over $ Z $) on a birational model of $ X_1 $ and let $ N_1 $ be an effective $ \R $-divisor on $ X_1 $ such that $ N_1 + P_1 $ is $ \R $-Cartier. Assume that the NQC g-pair $ \big( X_1, (B_1+N_1) + (M_1+P_1) \big) $ is lc and that the divisor $ K_{X_1} + B_1 + N_1 + M_1 + P_1 $ is nef over $ Z $.
		Consider a $(K_{X_1} + B_1 + M_1)$-MMP over $Z$ with scaling of $ N_1 + P_1 $:
		\begin{center}
			\begin{tikzcd}[column sep = 1em, row sep = large]
				(X_1,B_1+M_1) \arrow[dr, "\theta_1" swap] \arrow[rr, dashed, "\pi_1"] && (X_2,B_2+M_2) \arrow[dl, "\theta_1^+"] \arrow[dr, "\theta_2" swap] \arrow[rr, dashed, "\pi_2"] && (X_3,B_3+M_3) \arrow[dl, "\theta_2^+"] \arrow[rr, dashed, "\pi_3"] && \cdots \\
				& Z_1 && Z_2
			\end{tikzcd}
		\end{center}
		and denote by $N_i$ and $P_i$ the pushforwards of $N_1$ and $P_1$, respectively, to $X_i$. Then there exists a diagram
		\begin{center}
			\begin{tikzcd}[column sep = 1em, row sep = large]
				(X_1',B_1'+M_1') \arrow[d, "h_1" swap] \arrow[rr, dashed, "\rho_1"] && (X_2',B_2'+M_2') \arrow[d, "h_2" swap] \arrow[rr, dashed, "\rho_2"] && (X_3',B_3'+M_3') \arrow[d, "h_3" swap] \arrow[rr, dashed, "\rho_3"] && \dots 
				\\ 
				(X_1,B_1+M_1) \arrow[dr, "\theta_1" swap] \arrow[rr, dashed, "\pi_1"] && (X_2,B_2+M_2) \arrow[dl, "\theta_1^+"] \arrow[dr, "\theta_2" swap] \arrow[rr, dashed, "\pi_2"] && (X_3,B_3+M_3) \arrow[dl, "\theta_2^+"] \arrow[rr, dashed, "\pi_3"] && \dots \\
				& Z_1 && Z_2
			\end{tikzcd}
		\end{center}
		where, for each $i \geq 1$, 
		\begin{itemize}
			\item the map $ \rho_i \colon X_i' \dashrightarrow X_{i+1}' $ is a $(K_{X_i'}+B_i'+M_i')$-MMP over $Z$ with scaling of $ N_i' + P_i' $, where the divisors $ N_i' $ and $ P_i' $ on $ X_i' $ are defined as in \eqref{eq:2_lifting_MMP_with_scaling} (or as in \eqref{eq:1_lifting_MMP_with_scaling} for $i=1$), and
			
			\item the map $ h_i \colon (X_i',B_i'+M_i') \to (X_i,B_i+M_i) $ is a dlt blow-up.
		\end{itemize}
		In particular, the sequence on top of the above diagram is a $ (K_{X_1'}+B_1'+M_1') $-MMP over $ Z $ with scaling of $ N_1' + P_1' = h_1^* (N_1 + P_1) $, where $ (X_1'/Z,B_1'+M_1') $ is an NQC $ \Q $-factorial dlt g-pair.
		
		Furthermore, if we denote by $\lambda_i$ the nef thresholds at the steps of the $(K_{X_1} + B_1 + M_1)$-MMP over $Z$ with scaling of $ N_1 + P_1 $, that is,
		\[ \lambda_i := \inf \{ t \in \R_{\geq 0} \mid K_{X_i} + (B_i + t N_i) + (M_i + t P_i) \text{ is nef over } Z \} , \]
		and by $\mu_i$ the nef thresholds at the corresponding steps of the $ (K_{X_1'}+B_1'+M_1') $-MMP over $ Z $ with scaling of $ N_1' + P_1' = h_1^* (N_1 + P_1) $, that is,
		\[ \mu_i := \inf \{ t \in \R_{\geq 0} \mid K_{X_i'} + (B_i' + t N_i') + (M_i' + t P_i') \text{ is nef over } Z \} , \]
		then it holds that 
		\[ \lambda_i = \mu_i \quad \text{for every } i \geq 1 . \]
		
		Finally, there exists an integer $k \geq 1$ such that the map $ \pi_i \colon X_i \dashrightarrow X_{i+1} $ is small and the induced linear map $ N^1(X_i/Z)_\R \to N^1(X_{i+1}/Z)_\R $ is an isomorphism for each $ i \geq k $.
	\end{thm}
	
	We conclude this section with some clarifying comments about Theorem \ref{thm:lifting_MMP_with_scaling}. 
	Each map $ \rho_i \colon X_i' \dashrightarrow X_{i+1}' $ is, in general, the composite of finitely many steps of a $(K_{X_i'}+B_i'+M_i')$-MMP over $ Z $ with scaling of $ N_i' + P_i' = h_i^* (N_i + P_i) $, but not necessarily just one step (e.g., a flip). Additionally, the nef thresholds at the steps of this MMP, denoted by 
	$$ \nu_j^{(i)} , \ i \geq 1 , \ 1 \leq j \leq k_i , $$ 
	where $ \nu_1^{(i)}$ is computed on $X_i'$ and $ \nu_{k_i}^{(i)} $ is computed on $X_{i+1}'$, 
	satisfy the following properties:
	\[ \nu_j^{(i)} = \lambda_i \ \text{ for all } 1 \leq j < k_i \quad \text{and} \quad \nu_{k_i}^{(i)} = \lambda_{i+1} , \]
	where $ \lambda_i $ denotes the nef threshold at the $i$-th step of the given $(K_{X_1} + B_1 + M_1)$-MMP over $Z$ with scaling of $ N_1 + P_1 $. Finally, the sequence $ \{ \mu_i \}_{i=1}^{+\infty} $ of nef thresholds defined above forms a subsequence of $ \big\{ \nu_j^{(i)} \big\}_{i=1, \ j=1}^{+\infty, \ k_i} $, namely, we have
	\[ \mu_i = \nu_1^{(i)} = \lambda_i \quad \text{for each } i \geq 1 . \]

	\section{Proof of Theorem \ref{thm:mainthm_intro} and an application}
	
	We prove here our main result, Theorem \ref{thm:mainthm_intro}. This is accomplished by removing from \cite[Proposition A.3 and Theorem 1.2]{LT22b} the assumption that the underlying variety is $ \Q $-factorial with the aid of \cite[Theorem 1.5]{Xie22}; see Proposition \ref{prop:mainprop} and Theorem \ref{thm:mainthm}, respectively. The strategy for the proof of these two results is essentially the same as the one employed in \cite{LT22b}, so we only outline their proofs below and we refer to op.\ cit.\ for the details, although there are some additional complications now due to the absence of $\Q$-factoriality. 
	
	\begin{prop}\label{prop:mainprop}
		Let $ (X/Z,B+M) $ be an NQC lc g-pair. Assume that $ (X,B+M) $ has a minimal model in the sense of Birkar-Shokurov over $ Z $ or that $ K_X+B+M $ is not pseudo-effective over $ Z $. Then there exists a $ (K_X + B + M) $-MMP over $Z$ which terminates. In particular, $ (X,B+M) $ has a minimal model or a Mori fiber space over $ Z $.
	\end{prop}
	
	\begin{proof}
		To prove the statement, we follow closely the proofs of \cite[Proposition 6.2]{HH20} and \cite[Propositions 5.1 and A.3]{LT22b}. 
		
		First, since we work here without the assumption that $ X $ is $ \Q $-factorial, by applying Theorem \ref{thm:lifting_MMP} and by replacing $(X,B+M)$ with an appropriate g-pair, we may assume that any $(K_X+B+M)$-MMP over $Z$
		\[ (X_1, B_1+M_1) := (X,B+M) \dashrightarrow (X_2,B_2+M_2) \dashrightarrow \cdots \dashrightarrow (X_i,B_i+M_i)\dashrightarrow \cdots \]
		has the property that for each $i \geq 1$, the map $X_i \dashrightarrow X_{i+1}$ is small and $\rho(X_i/Z) = \rho(X_{i+1}/Z)$; see the second and third paragraph of the proof of \cite[Proposition 6.2]{HH20}, as well as the second paragraph of the proof of \cite[Proposition 5.1]{LT22b}.
		
		Next, by arguing by contradiction and by repeating essentially verbatim the proof of \cite[Proposition 5.1]{LT22b} (starting from the third paragraph of op.\ cit.\ and also replacing [14, Section 3.3] = \cite[Section 3.3]{HanLi22} with \cite[Theorem 1.4]{Chen23}), we infer that eventually there exists a $ (K_X+B+M) $-MMP over $ Z $ with scaling of an (appropriately chosen) ample divisor $ A $, which consists only of flips, whose nef thresholds satisfy $ \lambda_i > \lambda_{i+1} $ for every $ i \geq 1 $, and which does not terminate by assumption. We remark in passing that for the proof of the above strict inequality in our setting we invoke Lemma \ref{lem:properties_of_MMP}(i) as in the fifth paragraph of the proof of \cite[Proposition 6.2]{HH20} and we also need to replace [16, Theorem 1.3(4)(c)] = \cite[Theorem 1.1(5)(c)]{HL23} with (the third bullet of) \cite[Theorem 1.5]{Xie22}. 
		
		Finally, we set $\lambda := \lim_{i\to + \infty}\lambda_i$ and we observe that $\lambda < \lambda_i$ for every $ i \geq 1 $. Since each divisor $ K_{X_i} + B_i + M_i + \lambda_i A_i $ is nef over $ Z $ by construction of the MMP with scaling, where $(X_i,B_i+M_i)$ is the NQC lc g-pair appearing at the $ i $-th step of this MMP and $A_i$ is the strict transform of $A$ on $X_i$, the divisor $K_X+B+M+\lambda A$ is pseudo-effective over $Z$. By the assumptions of the proposition when $ \lambda = 0 $ or by \cite[Theorem 1.3(1)]{LX23a} when $ \lambda > 0 $ we conclude that the g-pair $\big(X,(B+\lambda A)+M\big)$ has a minimal model in the sense of Birkar-Shokurov over $Z$. Hence, the above MMP terminates by Theorem \ref{thm:LT22b_4.1}, a contradiction. 
	\end{proof}
	
	\begin{thm}\label{thm:mainthm}
		Let $ (X/Z,B+M) $ be an NQC lc g-pair. Assume that $ (X,B+M) $ has a minimal model in the sense of Birkar-Shokurov over $ Z $ or that $ K_X+B+M $ is not pseudo-effective over $ Z $. Let $ A $ be an effective $ \R $-Cartier $ \R $-divisor on $ X $ which is ample over $ Z $ such that $ \big( X/Z, (B+A)+M \big) $ is lc and $ K_X + B + A + M $ is nef over $Z$. Then there exists a $ (K_X + B + M) $-MMP over $Z$ with scaling of $A$ that terminates. In particular, $ (X,B+M) $ has a minimal model or a Mori fiber space over $ Z $.
	\end{thm}
	
	\begin{proof}
		We may repeat verbatim the proof of \cite[Theorem 1.2]{LT22b} (which is essentially the same as that of \cite[Theorem 5.2]{LT22b} with $P = A$ and $N = 0$), except that we replace \cite[Theorem A.2]{LT22b} with \cite[Theorem 1.3(1)]{LX23a}, \cite[Proposition 2.12]{LT22b} with Lemma \ref{lem:trivial_MMP}, \cite[Proposition A.3]{LT22b} with Proposition \ref{prop:mainprop}, and finally \cite[Theorem 4.1]{LT22b} with Theorem \ref{thm:LT22b_4.1}.
	\end{proof}
	
	We conclude this brief section with an application of Theorem \ref{thm:mainthm}, which generalizes \cite[Lemma 2.12]{Hash20d} to the setting of g-pairs and improves on \cite[Lemma 2.10]{LX23a} by removing the assumption that the underlying variety is $ \Q $-factorial klt. For brevity we only indicate below the necessary modifications to the proof of \cite[Lemma 2.10]{LX23a}.
	
	\begin{lem}\label{lem:scaling_numbers_to_zero}
		Let $(X/Z,B+M)$ be an NQC lc g-pair. Let $H$ be an effective $\R$-Cartier $\R$-divisor on $X$ such that the NQC g-pair $\big(X/Z, (B+H)+M \big)$ is lc and the divisor $K_X+B+H+M$ is nef over $Z$. Assume also that for any $\nu \in (0,1]$ the NQC lc g-pair $ \big(X, (B + \nu H) + M \big)$ has a minimal model in the sense of Birkar-Shokurov over $ Z $. Then we can construct a $(K_X+B+M)$-MMP over $Z$ with scaling of $H$
		\[ (X_1, B_1+M_1) := (X,B+M) \dashrightarrow (X_2,B_2+M_2) \dashrightarrow \cdots \dashrightarrow (X_i,B_i+M_i)\dashrightarrow \cdots \]
		with the following property: if we denote by $H_i$ the strict transform of $H$ on $X_i$ and by 
		$$ \lambda_i := \inf \{ t \geq 0 \mid K_{X_i}+B_i+tH_i+M_i \text{ is nef over } Z \} $$
		the corresponding nef threshold, then it holds that 
		\[ \lim_{i \to +\infty}\lambda_i=0 , \]
		regardless of whether this $ (K_X+B+M) $-MMP over $ Z $ with scaling of $H$ terminates or not.
	\end{lem}
	
	\begin{proof}
		We may repeat verbatim the proof of \cite[Lemma 2.10]{LX23a}, except that we replace \cite[Lemma 3.21]{HanLi22} with Lemma \ref{lem:trivial_MMP}, [15, Theorem 2.24] = \cite[Theorem 2.8]{HL23} with Theorem \ref{thm:mainthm}, and \cite[Theorem 4.1]{HanLi22} with Theorem \ref{thm:LT22b_4.1}.
	\end{proof}

	\section{Applications -- Part I}
	\label{section:applications_I}
	
	In this section we present numerous applications of Proposition \ref{prop:mainprop}. The title of each subsection indicates clearly its contents.

	\subsection{NQC weak Zariski decompositions}
	\label{subsection:NQC_WZD}
	
	We first recall the notion of an NQC weak Zariski decomposition and we refer to \cite{Bir12b,HanLi22,LT22a,LT22b} for more information. Afterwards, we prove Theorem \ref{thm:MM_NQC_WZD_intro} and Corollary \ref{cor:maincor_intro}(iii).
	
	\begin{dfn}
		Let $ X \to Z $ be a projective morphism between normal varieties and let $ D $ be an $ \R $-Cartier $\R$-divisor on $ X $. An \emph{NQC weak Zariski decomposition of $ D $ over $Z$} consists of a projective birational morphism $ f \colon W \to X $ from a normal variety $ W $ and a numerical equivalence $ f^* D \equiv_Z P + N $, where $ P $ is an NQC divisor (over $ Z $) on $ W $ and $ N $ is an effective $\R$-Cartier $\R$-divisor on $W$.
		
		Let $ (X/Z,B+M) $ be an NQC g-pair such that $ K_X+B+M $ is pseudo-effective over $ Z $. We say that $ (X,B+M) $ \emph{admits an NQC weak Zariski decomposition over $Z$} if the divisor $ K_X + B + M $ admits an NQC weak Zariski decomposition over $Z$.
	\end{dfn}
	
	The state-of-the-art result concerning the relation between the existence of NQC weak Zariski decompositions and the existence of minimal models for generalized pairs is the following:
	
%
%
%
%
	
	\begin{thm}\label{thm:MM_NQC_WZD}
		Assume the existence of minimal models for smooth varieties of dimension $n-1$.
		
		Let $ (X/Z,B+M) $ be an NQC lc g-pair of dimension $ n $. The following are equivalent:
		\begin{enumerate}[\normalfont (i)]
			\item $ (X,B+M) $ admits an NQC weak Zariski decomposition over $Z$,
			
			\item $ (X,B+M) $ has a minimal model over $Z$.
		\end{enumerate}
	\end{thm}
	
%
%
%
%

	\begin{proof}
		For the implication (ii) $ \implies $ (i) see, for example, \cite[Corollary 3.10]{Tsak21}. The converse implication follows immediately from \cite[Theorem 4.4(i)]{LT22a} and Proposition \ref{prop:mainprop}.
	\end{proof}
	
	\begin{cor}\label{cor:maincor_I}
		Let $ (X/Z,B+M) $ be an NQC lc g-pair such that $ K_X+B+M$ is pseudo-effective over $ Z $. If $ \dim X = 5$ and if $ (X,B+M) $ admits an NQC weak Zariski decomposition over $ Z $, then $ (X,B+M) $ has a minimal model over $Z$.
	\end{cor}
	
	\begin{proof}
		Follows immediately from \cite[Theorem 5-1-15]{KMM87} and Theorem \ref{thm:MM_NQC_WZD}.
	\end{proof}

	\subsection{On the existence of minimal models of generalized pairs}
	
	With the aid of results from \cite[Section 4]{LT22a}, we derive here several corollaries of Proposition \ref{prop:mainprop}, which refine \cite[Corollary 1.3, Theorem 1.5 and Corollary 1.4]{LT22b}, respectively, and which include \cite[Theorem A, Theorem C and Corollary D]{LT22a}, respectively, as special cases.
	
	\begin{thm} \label{thm:EMM_implication}
		The existence of minimal models for smooth varieties of dimension $n$ implies the existence of minimal models for NQC lc g-pairs of dimension $n$.
	\end{thm}
	
	\begin{proof}
		Follows immediately from \cite[Theorem 4.1(i)]{LT22a} and Proposition \ref{prop:mainprop}.
	\end{proof}
	
	\begin{thm}\label{thm:MM_uniruled}
		Assume the existence of minimal models for smooth varieties of dimension $n-1$.
		
		Let $ (X/Z,B+M) $ be an NQC lc g-pair of dimension $ n $ such that $ K_X+B+M $ is pseudo-effective over $ Z $. If a general fiber of the morphism $ X \to Z $ is uniruled, then $(X,B+M)$ has a minimal model over $Z$.
	\end{thm}
	
	\begin{proof}
		Follows immediately from \cite[Theorem 4.3]{LT22a} and Proposition \ref{prop:mainprop}.
	\end{proof}
	
	\begin{cor}\label{cor:maincor_II}
		Let $ (X/Z,B+M) $ be an NQC lc g-pair of dimension $ n $ such that $ K_X+B+M$ is pseudo-effective over $ Z $. The following statements hold:
		\begin{enumerate}[\normalfont (i)]
			\item If $n \leq 4$, then $ (X,B+M) $ has a minimal model over $Z$.
			
			\item If $n \leq 5$ and a general fiber of the morphism $ X\to Z $ is uniruled, then $ (X,B+M) $ has a minimal model over $Z$.
		\end{enumerate}
	\end{cor}
	
	\begin{proof}
		The existence of minimal models for terminal varieties of dimension $ n \leq 4 $ over $ Z $ is well-known; see \cite[Theorem 5-1-15]{KMM87}. Consequently, (i) and (ii) follow from Theorem \ref{thm:EMM_implication} and Theorem \ref{thm:MM_uniruled}, respectively.
	\end{proof}

	\subsection{Generalized pairs of maximal Albanese dimension}
	\label{subsection:mAd}
	
	In this subsection we mainly work in the absolute setting; see the beginning of Subsection \ref{subsection:log_abundant}. First, we briefly discuss the notion of maximal Albanese dimension for smooth and singular varieties.
	
	\begin{dfn}
		Let $X$ be a smooth projective variety. Denote by $\Alb(X)$ the Albanese variety of $X$ and let $\alpha \colon X \to \Alb(X)$ be the associated Albanese morphism. We say that $X$ has \emph{maximal Albanese dimension} if $\dim \alpha(X) = \dim X$.
	\end{dfn}
	
	The property of having maximal Albanese dimension is birationally invariant for smooth projective varieties. We provide the proof of this fact for the benefit of the reader.
	
	\begin{lem}\label{lem:bir_inv_smooth}
		If $ X $ and $ Y $ are birationally equivalent smooth projective varieties, then $X$ has maximal Albanese dimension if and only if $Y$ has maximal Albanese dimension.
	\end{lem}
	
	\begin{proof}
		By considering a resolution of indeterminacies $ (p,q) \colon W \to X \times Y $ of the birational map $ X \dashrightarrow Y $ such that $W$ is smooth, we immediately see that it suffices to prove the statement when $ X \dashrightarrow Y $ is actually a birational morphism.
		Denote it by $f$, and let $\beta \colon Y \to \Alb(Y)$ be the Albanese morphism of $Y$. By \cite[Proposition 9.12]{Ueno75}, the map $\alpha := \beta \circ f$ is the Albanese morphism of $X$. In particular,
		$ \alpha(X) = (\beta \circ f)(X) = \beta(Y) $,
%
		and since $\dim X = \dim Y$, it follows readily that $X$ has maximal Albanese dimension if and only if $Y$ has maximal Albanese dimension.
	\end{proof}
	
	\begin{dfn}
		Let $X$ be a projective variety. We say that $X$ has \emph{maximal Albanese dimension} if there exists a resolution $ W \to X$ of $X$ such that $W$ has maximal Albanese dimension.
	\end{dfn}
	
	According to Lemma \ref{lem:bir_inv_smooth}, the above definition does not depend on the choice of resolution of $X$. Additionally, the property of having maximal Albanese dimension is birationally invariant for singular varieties as well. More precisely:
	
	\begin{lem}\label{lem:bir_inv_singular}
		If $ X $ and $ Y $ are birationally equivalent projective varieties, then $X$ has maximal Albanese dimension if and only if $Y$ has maximal Albanese dimension.
	\end{lem}
	
	\begin{proof}
		The statement follows by considering a common resolution $ (p,q) \colon W \to X \times Y $ of $X$ and $Y$ and by invoking Lemma \ref{lem:bir_inv_smooth}.
	\end{proof}
	
	We now turn to the proofs of Theorems \ref{thm:mAd_klt_intro} and \ref{thm:mAd_lc_intro}, beginning with the former. For its proof we follow closely Fujino's strategy from \cite{Fuj13}, so we first derive analogues of \cite[Lemmas 3.1, 3.2 and 3.3]{Fuj13} in the context of g-pairs.
	
	\begin{lem}\label{lem:negative_rational_curves}
		Let $(X,B+M)$ be an NQC lc g-pair. If $K_X+B+M$ is not nef, then there exists a rational curve $C$ on $X$ such that $-2 \dim X \leq (K_X+B+M) \cdot C < 0$.
	\end{lem}
	
	\begin{proof}
		Follows immediately from \cite[Theorem 1.1(1)(2)]{HL23}.
	\end{proof}
	
	\begin{lem}\label{lem:no_rational_curves}
		Let $(X,B+M)$ be an NQC lc g-pair. Let $ g \colon X \to S $ be a morphism between projective varieties. Assume that $ K_X+B+M $ is nef over $ S $ and that $S$ contains no rational curves. Then $K_X+B+M$ is nef.
	\end{lem}
	
	\begin{proof}
		Assume, by contradiction, that $K_X+B+M$ is not nef. Then, by Lemma \ref{lem:negative_rational_curves}, there exists a rational curve $C$ on $X$ such that $(K_X+B+M) \cdot C < 0$. Since $K_X+B+M$ is nef over $S$ by assumption, the curve $C$ cannot be contracted by $g$, so its image $g(C) $ is a rational curve in $S$. However, this contradicts the assumption that $S$ does not contain any rational curves, and thus proves the statement.
	\end{proof}
	
	\begin{lem}\label{lem:relMM_vs_MM}
		Let $(X,B+M)$ be an NQC lc g-pair. Let $ g \colon X \to S $ be a morphism between projective varieties and assume that $S$ contains no rational curves. If $(X',B'+M')$ is a minimal model of $(X,B+M)$ over $S$, then $(X',B'+M')$ is a minimal model of $(X,B+M)$.
	\end{lem}
	
	\begin{proof}
		Follows immediately from the definition of minimal models and Lemma \ref{lem:no_rational_curves}.
	\end{proof}
	
	We are now ready to generalize \cite[Theorem 3.4]{Fuj13} to the setting of g-pairs.
	
	\begin{thm}\label{thm:mAd_klt}
		Let $ (X,B+M) $ be an NQC klt g-pair. If $ X $ has maximal Albanese dimension, then $ (X,B+M) $ has a minimal model.
	\end{thm}
	
	\begin{proof}
		Let $\pi \colon W \to X$ be a resolution of $X$ and let $ \beta \colon W \to \Alb(W)$ be the Albanese morphism of $W$. Since $X$ has rational singularities by \cite[Lemma 4.2(i)]{HanLiu20}, it follows from \cite[Lemma 2.4.1]{BS95book} that there exists a morphism $ g \colon X \to \Alb(W) $ such that the following diagram commutes:
		\begin{center}
			\begin{tikzcd}[row sep = 3em, column sep = 3em]
				W \arrow[d, "\pi" swap] \arrow[rr, "\beta"] && \Alb(W) \\
				X \arrow[urr, "g" swap]
			\end{tikzcd}
		\end{center}
		Set $S := g(X) = \beta(W) \subseteq \Alb(W)$ and note that $S$ contains no rational curves, since the same holds for $\Alb(W)$ by \cite[Proposition 4.9.5]{BLbook}. Moreover, the morphism $ g \colon X \to S $ is generically finite, since $X$ has maximal Albanese dimension, and thus $K_X+B+M$ is big over $S$; see \cite[p.\ 69, Remark (2)]{Nak04}. Hence, the statement follows from Theorem \ref{thm:EMM_general_type}(ii) and Lemma \ref{lem:relMM_vs_MM}.
	\end{proof}
	
	We recall that if $M=0$ in the previous theorem, then the (usual) klt pair $(X,B)$ has a \emph{good} minimal model; see \cite[Theorem 4.3]{Fuj13}. One may thus wonder whether a similar statement also holds in the category of generalized pairs, but one quickly realizes that the answer is negative in general. Indeed, Example \ref{exa:counter-example} indicates that NQC klt generalized pairs $(X,B+M)$ whose underlying variety $X$ has maximal Albanese dimension need not be abundant, and hence need not have good minimal models in general.
	
	Next, under mild assumptions in lower dimensions and by utilizing the close relation between the existence of NQC weak Zariski decompositions and the existence of minimal models, namely Theorem \ref{thm:MM_NQC_WZD}, we derive the following version of \cite[Theorem 3.4]{Fuj13} for NQC lc g-pairs.
	
	\begin{thm}\label{thm:mAd_lc}
		Assume the existence of minimal models for smooth varieties of dimension $n-1$.
		
		Let $ (X,B+M) $ be an NQC lc g-pair of dimension $ n $ such that $K_X+B+M$ is pseudo-effective. If $ X $ has maximal Albanese dimension, then $ (X,B+M) $ has a minimal model.
	\end{thm}
	
	\begin{proof}
		According to Remark \ref{rem:EMM_reduction} and Lemma \ref{lem:bir_inv_singular}, to prove the statement, we may assume that $(X,B+M)$ is $\Q$-factorial dlt. We distinguish two cases.
		
		\medskip
		
		\emph{Case I}: Assume that $\lfloor B \rfloor = 0$. Then the g-pair $(X,B+M)$ has klt singularities, so it has a minimal model by Theorem \ref{thm:mAd_klt}.
		
		\medskip
		
		\emph{Case II}: Assume that $\lfloor B \rfloor \neq 0$. We consider the quantity
		\[ \tau := \inf \left\{ t \in \R_{\geq 0} \mid K_X + \big(B - \lfloor B \rfloor \big) + t \lfloor B \rfloor +M \text{ is pseudo-effective} \right\} \in [0,1] \]
		and we distinguish two (sub)cases.
		
		Assume first that $\tau = 1$. Then $(X,B+M)$ admits a weak Zariski decomposition by \cite[Theorem 3.1]{LT22b}, so $(X,B+M)$ has a minimal model by Theorem \ref{thm:MM_NQC_WZD}.
		
		Assume now that $\tau \in [0,1)$. Set $\Delta := B-(1-\tau)\lfloor B \rfloor $ and note that, by construction, $(X,\Delta+M)$ is an NQC $\Q$-factorial klt g-pair such that $K_X+\Delta+M$ is pseudo-effective. Therefore, $(X,\Delta+M)$ has a minimal model by Theorem \ref{thm:mAd_klt}, so it admits a weak Zariski decomposition by Theorem \ref{thm:MM_NQC_WZD}. Since $B = \Delta + (1-\tau) \lfloor B \rfloor$ and since $X$ is $\Q$-factorial, it is straightforward that $(X,B+M)$ also admits a weak Zariski decomposition, and hence it has a minimal model by Theorem \ref{thm:MM_NQC_WZD}.
	\end{proof}
	
	\begin{cor}\label{cor:mAd_lc_dim5}
		Let $ (X,B+M) $ be an NQC lc g-pair such that $ K_X+B+M$ is pseudo-effective. If $ \dim X = 5$ and if $X$ has maximal Albanese dimension, then $ (X,B+M) $ has a minimal model.
	\end{cor}
	
	\begin{proof}
		Follows immediately from \cite[Theorem 5-1-15]{KMM87} and Theorem \ref{thm:mAd_lc}.
	\end{proof}
	
	Finally, for the sake of completeness, we remark that \cite[Theorem 5.1]{Fuj13} can also be generalized to the setting of g-pairs. More precisely, using the fact that lc centers of dlt g-pairs are normal (see \cite[Lemma 2.9 and Subsection 2.3]{HanLi22}) and invoking Lemma \ref{lem:negative_rational_curves} and \cite[Theorem 1.1]{Xie22}, we may prove, arguing as in the proof of \cite[Theorem 5.1]{Fuj13}, the following statement: if $(X,B+M)$ is an NQC dlt g-pair such that $K_X+B+M$ is log big with respect to $(X,B+M)$ and if $X$ contains no rational curves, then $K_X+B+M$ is ample. We leave the details of the proof to the interested reader, since this result will not be needed in this paper.

	\subsection{NQC Nakayama-Zariski decomposition}
	\label{subsection:NQC_NZD}
	
	In this subsection we work exclusively in the absolute setting; see the beginning of Subsection \ref{subsection:log_abundant}.
	
	Given a pseudo-effective $\R$-divisor $D$ on a smooth projective variety $X$, Nakayama \cite{Nak04} defined a decomposition $ D = P_\sigma (D) + N_\sigma (D) $, which is usually called \emph{the Naka\-yama-Zariski decomposition of $D$}. The divisor $ P_\sigma(D) $ (resp.\ $ N_\sigma(D) $) is called \emph{the positive part} (resp.\ \emph{the negative part}) of the Nakayama-Zariski decomposition of $D$. Note that $ N_\sigma(D) $ is effective by construction and $ P_\sigma(D) $ is movable by \cite[Lemma III.1.8 and Proposition III.1.14(1)]{Nak04}. For general properties of the Nakayama-Zariski decomposition we refer to \cite[Chapter III]{Nak04} and \cite[Lemma 4.1]{BH14b}.
	
	The above decomposition can be extended both to the singular setting, see for instance \cite[Section 4]{BH14b} and \cite[Subsection 2.1]{Hash20d}, and to the relative setting, see for example \cite[Subsection III.4]{Nak04} and \cite[Section 3]{LX23a}. However, according to \cite{Les16}, the relative Nakayama-Zariski decomposition of a relatively pseudo-effective $ \R $-divisor is not always well-defined.
	
	We now recall the following definitions.
	
	\begin{dfn} \label{dfn:NQC_NZD}
		Let $ X $ be a normal projective variety and let $ D $ be a pseudo-effective $ \R $-Cartier $ \R $-divisor on $ X $. We say that \emph{$ D $ admits birationally a Nakayama-Zariski decomposition with nef} (resp.\ \emph{NQC}, \emph{semi-ample}) \emph{positive part} if there exists a resolution $ f \colon W \to X $ such that $ P_\sigma ( f^* D ) $ is nef (resp.\ NQC, semi-ample).
		
		Let $ (X,B+M) $ be a g-pair such that $ K_X+B+M $ is pseudo-effective. We say that $ (X,B+M) $ \emph{admits birationally a Nakayama-Zariski decomposition with nef} (resp.\ \emph{NQC}, \emph{semi-ample}) \emph{positive part} if the divisor $ K_X + B + M $ admits birationally a Nakayama-Zariski decomposition with nef (resp.\ NQC, semi-ample) positive part.
	\end{dfn}
	
	If the given g-pair $ (X,B+M) $ is NQC, 
	then we will be interested only in a birational Nakayama-Zariski decomposition of $ K_X+B+M $ with NQC positive part. For brevity we sometimes refer to such a decomposition simply as an \emph{NQC Nakayama-Zariski decomposition}, since it is highly unlikely that this will cause any confusion. The NQC variant of Definition \ref{dfn:NQC_NZD} was introduced in \cite{Tsak21}, where various properties of that decomposition were established, which are completely analogous to properties of weak Zariski decompositions that were thoroughly studied in \cite[Subsection 2.3]{LT22a}. We refer to \cite[Chapter 3]{Tsak21} for further information.
	
	\medskip
	
	The next theorem describes the relation between the existence of \enquote{good} Nakayama-Zariski decompositions and the existence of (good) minimal models for (generalized) pairs, cf.\ Theorem \ref{thm:MM_NQC_WZD}. Part (i) is essentially \cite[Theorem 1.1]{BH14b}, while part (ii) is essentially \cite[Theorem 4.18]{Tsak21}. For brevity we only outline the proof below.
	
	\begin{thm} \label{thm:MM_NQC_NZD}
		The following statements hold:
		
		\begin{enumerate}[\normalfont (i)]
			\item Let $ (X,B) $ be an lc pair such that $ K_X+B $ is pseudo-effective. Then $ (X,B) $ has a minimal model (resp.\ good minimal model) if and only if it admits birationally a Nakayama-Zariski decomposition with nef (resp.\ semi-ample) positive part.
			
			\item Let $ (X,B+M) $ be an NQC lc g-pair such that $ K_X+B+M $ is pseudo-effective. Then $ (X,B+M) $ has a minimal model (resp.\ good minimal model) if and only if it admits birationally a Nakayama-Zariski decomposition with NQC (resp.\ semi-ample) positive part.
		\end{enumerate}
	\end{thm}
	
	\begin{proof}~
		
		\medskip
		
		\noindent (i) If $ (X,B) $ has a minimal model (resp.\ good minimal model), then it follows from \cite[Remark 2.6]{Bir12a} and \cite[Lemma 4.1(2)]{BH14b} that $ (X,B) $ admits birationally a Nakayama-Zariski decomposition with nef (resp.\ semi-ample) positive part. 
		
		If $ (X,B) $ admits birationally a Nakayama-Zariski decomposition with nef positive part, then it has a minimal model in the sense of Birkar-Shokurov by \cite[Theorem 1.1]{BH14b}, and thus it has a minimal model by \cite[Theorem 1.7]{HH20}. If, moreover, $ (X,B) $ admits birationally a Nakayama-Zariski decomposition with semi-ample positive part, then the same arguments also yield that $ (X,B) $ has a good minimal model, taking into account \cite[Theorem 3.25(4)]{KM98} and Lemma \ref{lem:GMM}; see also (the paragraph preceeding) \cite[Theorem 2.23]{Hash20d}. 
		
		\medskip
		
		\noindent (ii) The \enquote{only if} part of the statement follows from (the proof of) \cite[Corollary 3.27]{Tsak21}. As far as the \enquote{if} part of the statement is concerned, if $ (X,B+M) $ admits birationally a Nakayama-Zariski decomposition with NQC positive part, then it has a minimal model in the sense of Birkar-Shokurov by \cite[Theorem 4.18]{Tsak21}, and thus it has a minimal model by Proposition \ref{prop:mainprop}. If, moreover, $ (X,B+M) $ admits birationally a Nakayama-Zariski decomposition with semi-ample positive part, then by repeating verbatim the proof of \cite[Theorem 4.18]{Tsak21}, while bearing \cite[2.1.8]{Fuj17book} in mind, we deduce that $ (X,B+M) $ has a good minimal model in the sense of Birkar-Shokurov. It follows from Lemma \ref{lem:GMM} and Proposition \ref{prop:mainprop} that $ (X,B+M) $ has a good minimal model.
	\end{proof}
	
	\begin{rem}~
		\label{rem:comment_BH14b_1.1-g}
		\begin{enumerate}[(1)]
			\item In Theorem \ref{thm:MM_NQC_NZD}(i) there is no loss of generality if we replace the phrase \enquote{$ (X,B) $ admits birationally a Nakayama-Zariski decomposition with nef positive part} with the phrase \enquote{$ (X,B) $ admits birationally a Nakayama-Zariski decomposition with NQC positive part}. Indeed, this follows readily from the construction and \cite[2.1.8 and Theorem 4.7.2(3)]{Fuj17book}. Thus, (i) may be regarded as a special case of (ii) for $ M = 0 $.
			
			\item The NQC condition is used crucially in Theorem \ref{thm:MM_NQC_NZD}(ii) due to the application of \cite[Lemma 3.22]{HanLi22} for its proof.
		\end{enumerate}
	\end{rem}
	
	We conclude this subsection with the generalization of \cite[Theorem 1.5]{Hash22c} to the context of g-pairs. More precisely, we establish the minimal model theory for NQC lc g-pairs admitting an lc-trivial fibration with log big moduli part (noting that log bigness is satisfied on a sufficiently high birational model of the base of the lc-trivial fibration). We only outline the proof below and we refer to op.\ cit.\ for the details as well as for the relevant definitions.
	
%
%

%
%
%
%
%
	
	\begin{thm}\label{thm:MM_logbigmoduli}
		Let $(X,B+M)$ be an NQC lc g-pair admitting an lc-trivial fibration with log big moduli part. Then $(X,B+M)$ has a good minimal model or a Mori fiber space.
	\end{thm}

	\begin{proof}
		Assume first that $ K_X+B+M $ is not pseudo-effective. Then $ (X,B+M) $ has a Mori fiber space by Proposition \ref{prop:mainprop}.
			
		Assume now that $K_X+B+M$ is pseudo-effective. In order to show that $(X,B+M)$ has a good minimal model, we argue essentially as in the proof of \cite[Theorem 1.5]{Hash22c},
		except that we now invoke Theorem \ref{thm:MM_NQC_NZD}(ii) instead of \cite[Theorem 2.23]{Hash20d}.
	\end{proof}

	\subsection{Generalized pairs of numerical dimension zero}
	\label{subsection:num_dim_zero}
	
	In this subsection we work exclusively in the absolute setting; see the beginning of Subsection \ref{subsection:log_abundant}. We prove here the existence of minimal models of NQC lc g-pairs of numerical dimension zero. To this end, we first derive an analogue of \cite[Theorem 5.1]{Gon11} in the context of g-pairs.
	
	\begin{lem}\label{lem:Gon11_Thm5.1}
		Let $ (X,B+M) $ be an NQC lc g-pair such that $ (X,0) $ is $ \Q $-factorial klt and $ \kappa_\sigma(X,K_X+B+M) = 0 $. Let $A$ be an effective ample $\R$-divisor on $X$ such that $ \big(X, (B+A)+M \big) $ is lc and $ K_X+B+A+M $ is nef. Then any $ (K_X+B+M) $-MMP with scaling of $ A $ terminates.
	\end{lem}
	
	\begin{proof}
		We run a $ (K_X+B+M) $-MMP with scaling of $ A $. For each $ i \geq 1 $, we denote by $(X_i,B_i+M_i)$ the g-pair appearing at the $ i $-th step of this MMP, where $ (X_1,B_1+M_1) := (X,B+M) $, and by $ \lambda_i $ the corresponding nef threshold. We set $\lambda :=\lim_{i \to +\infty} \lambda_i$ and we distinguish two cases. 
		
		First, if $ \lambda > 0 $, then we claim that the given MMP terminates.
		Indeed, this MMP is also a $(K_X + B + M + \frac{\lambda}{2} A)$-MMP. By \cite[Lemma 3.4]{HanLi22} there exists a boundary $ \Delta $ on $ X $ such that 
		$K_X + \Delta \sim_\R K_X + B + M + \frac{\lambda}{2} A  $, 
		$(X,\Delta)$ is klt and $ \Delta $ is big. By \cite[Corollary 1.4.2]{BCHM10}, the $(K_X+\Delta)$-MMP with scaling of $ A $ over $ Z $ terminates, and therefore the original MMP terminates.
		
		Second, if $ \lambda = 0 $, then we will also show that the given MMP terminates. To this end, arguing by contradiction, we assume that the above $ (K_X+B+M) $-MMP with scaling of $ A $ does not terminate. We may also assume that it consists only of flips. By \cite[Lemma 2.17]{LT22a} there exists an index $ \ell \geq 1 $ such that the divisor $ K_{X_\ell} + B_\ell + M_\ell $ is movable (see \cite[Definition III.1.13]{Nak04}). On the other hand, it follows from the Negativity lemma \cite[Lemma 3.39(1)]{KM98} and from \cite[Remark 2.15(2)]{Hash20d} that $ \kappa_\sigma(X_\ell, K_{X_\ell}+B_\ell+M_\ell) = 0 $. Therefore, we may replace $ (X,B+M) $ with $ (X_k, B_k+M_k) $ and we may thus assume that the divisor $ K_X+B+M $ is movable.
		
		Next, let $ f \colon W \to X $ be a log resolution of $ (X,B) $ and consider the Nakayama-Zariski decomposition of $ f^*(K_X+B+M) $:
		\begin{equation}\label{eq:1_Gon11_Thm5.1}
			f^* (K_X+B+M) = P_\sigma \big( f^* (K_X+B+M) \big) + N_\sigma \big( f^* (K_X+B+M) \big) .
		\end{equation}
		Since by \cite[Remark 2.15(2)]{Hash20d} we have
		\[ \kappa_\sigma \big(W, f^* (K_X+B+M) \big) = \kappa_\sigma(X,K_X+B+M) = 0 , \]
		by \cite[Proposition V.2.7(8)]{Nak04} we infer that 
		\begin{equation}\label{eq:2_Gon11_Thm5.1}
			P_\sigma \big( f^* (K_X+B+M) \big) \equiv 0 .
		\end{equation} 
		We now claim that $ N_\sigma \big( f^* (K_X+B+M) \big) \geq 0 $ is an $ f $-exceptional divisor. Indeed, there would otherwise exist a component $ G $ of $ N_\sigma \big( f^* (K_X+B+M) \big) $ which would not be $ f $-exceptional, so $ f_*G $ would be a component of $ N_\sigma(K_X+B+M) $ by \cite[Theorem III.5.16]{Nak04}, but this is impossible by \cite[Proposition III.1.14(1)]{Nak04}. (Alternatively, one may argue as in the proof of \cite[Claim 5.2]{Gon11}, taking \cite[Lemma III.1.4(5)]{Nak04} and \cite[Lemma 2.4]{Hash20d} into account.) Therefore, the previous claim, together with \eqref{eq:1_Gon11_Thm5.1} and \eqref{eq:2_Gon11_Thm5.1}, imply that $ K_X+B+M \equiv 0 $; in particular, $ K_X+B+M $ is nef. However, this contradicts our assumption that the given $(K_X+B+M)$-MMP with scaling of $A$ does not terminate, and finishes the proof.
	\end{proof}
	
	We are now ready to generalize \cite[Theorem 1.1 = Corollary 5.1]{Gon11} to the setting of g-pairs. Even though part (i) of the next theorem is a special case of part (ii), for the sake of completeness we also give a direct proof of (i) which depends only on \cite{Gon11} and \cite{HH20}.
	
	\begin{thm}\label{thm:Gon11_Thm1.1}
		The following statements hold:
		\begin{enumerate}[\normalfont (i)]
			\item If $ (X,B) $ is an lc pair such that $ \kappa_\sigma(X,K_X+B) = 0 $, then $ (X,B) $ has a minimal model.
			
			\item If $ (X,B+M) $ is an NQC lc g-pair such that $ \kappa_\sigma(X,K_X+B+M) = 0 $, then $ (X,B+M) $ has a minimal model.
		\end{enumerate}
	\end{thm}
	
	\begin{proof}~
		
		\medskip
		
		\noindent (i) Let $ h \colon (T,B_T) \to (X,B) $ be a dlt blow-up of $ (X,B) $. By \cite[Remark 2.15(2)]{Hash20d} it holds that $ \kappa_\sigma(T,K_T+B_T) = 0 $, so $ (T,B_T) $ has a minimal model $ (Y,B_Y) $ by \cite[Corollary 5.1]{Gon11}. Note that $ (Y,B_Y) $ is a minimal model in the sense of Birkar-Shokurov of $ (X,B) $. We conclude by \cite[Theorem 1.7]{HH20}.
		
		\medskip
		
		\noindent (ii) Let $ h \colon (T,B_T+M_T) \to (X,B+M) $ be a dlt blow-up of $ (X,B+M) $. By \cite[Remark 2.15(2)]{Hash20d} it holds that $ \kappa_\sigma(T,K_T+B_T+M_T) = 0 $, so $ (T,B_T+M_T) $ has a minimal model $ (Y,B_Y+M_Y) $ by Lemma \ref{lem:Gon11_Thm5.1}. Note that $ (Y,B_Y+M_Y) $ is a minimal model in the sense of Birkar-Shokurov of $ (X,B+M) $. We conclude by Proposition \ref{prop:mainprop}.
	\end{proof}
	
	Finally, one may wonder whether it is also possible to extend \cite[Theorem 1.2 = Theorem 6.1]{Gon11} to the setting of g-pairs; in other words, whether NQC lc g-pairs of numerical dimension zero are abundant. The following example, which was first discussed in \cite[p.\ 212, Nonvanishing]{BH14b}, demonstrates that this fails in general.
	
	\begin{exa}\label{exa:counter-example}
		Let $ X $ be an elliptic curve, set $ B := 0 $ and take $ M $ to be a non-torsion divisor on $ X $ of degree zero. Then $ (X,B+M) $ is an NQC lc g-pair such that
		\[ \kappa_\sigma(X,K_X+B+M) = \kappa_\sigma(X,M) = 0  \]
		and
		\[ \kappa_\iota(X,K_X+B+M) = \kappa(X,M) = - \infty . \]
	\end{exa}

	\subsection{Generalized pairs whose boundary contains an ample divisor}
	
	We establish here an analogue of \cite[Theorem 1.5]{HH20} in the context of g-pairs, which also refines \cite[Theorem 1.3(2)]{LX23a} significantly; see Theorem \ref{thm:EGMM_boundary_contains_ample}. To this end, we first prove two auxiliary results. The first one is the relative version of \cite[Proposition 1.45]{KM98}, so it should be well-known, but we provide its proof for the convenience of the reader nonetheless.
	
	\begin{lem}\label{lem:KM98_1.45_relative}
		Let $ f \colon X \to Y $ and $ g \colon Y \to Z $ be projective morphisms of varieties. If $ H $ is an $f$-ample divisor on $X$ and if $ L $ is a $g$-ample divisor on $Y$, then the divisor $ H + \nu f^* L$ is $(g \circ f)$-ample for $ \nu \gg 0 $.
	\end{lem}
	
	\begin{proof}
		Fix a point $ z \in Z $ and consider the projective fibers $ Y_z := g^{-1}(z) $ and $ X_z := (g \circ f)^{-1}(z) = f^{-1}(Y_z) $ and the induced map $ f |_{X_z} \colon X_z \to Y_z $. Note that $ L |_{Y_z} $ is ample and $ H |_{X_z} $ is ample over $ Y_z $. Therefore, by \cite[Proposition 1.45]{KM98}, $ (H + \nu_z f^*L) |_{X_z} $ is ample on $X_z$ for $ \nu_z \gg 0 $. It follows now from \cite[Proposition 1.41]{KM98} that the divisor $ H + \nu_z f^*L $ is ample over some open neighborhood $ U_z \subseteq Z $ of $z$ for $ \nu_z \gg 0 $. Since $ Z $ is quasi-compact, it can be covered by only finitely many such open subsets $ U_z \subseteq Z $. Thus, for any sufficiently large positive integer $ \nu $, the divisor $ H + \nu f^*L $ on $X$ is ample over $ Z $, as asserted.
	\end{proof}
	
	\begin{lem}\label{lem:MMP_ampleness_boundary}
		Let $\big(X/Z, (B+A)+M \big)$ be an NQC lc g-pair, where $A$ is an effective $ \R $-Cartier $ \R $-divisor on $ X $ which is ample over $Z$ and contains no lc center of the g-pair $ \big(X/Z,(B+A)+M \big)$. If $ \varphi \colon X \dasharrow X' $ is a partial $ (K_X+B+A+M) $-MMP over $Z$, then there exist effective $ \R $-divisors $ \widetilde A $ and $ \widetilde B $ on $ X' $ such that
		\begin{itemize}
			\item $ \widetilde A $ is ample over $ Z $,
			
			\item $ \widetilde B + \widetilde A \sim_{\R,Z} \varphi_*(B+A)$,
			
			\item $ \big( X'/Z, (\widetilde B + \widetilde A) + M' \big) $ is an NQC lc g-pair, and
			
			\item $\widetilde A$ contains no lc center of $ \big( X'/Z, (\widetilde B + \widetilde A) +M' \big) $.
		\end{itemize}
	\end{lem}
	
	\begin{proof}
		We may assume that $\varphi$ is a single step of a $(K_X+B+A+M)$-MMP over $Z$:
		\begin{center}
			\begin{tikzcd}
				(X,B+M) \arrow[rr, dashed, "\varphi"] \arrow[dr, "g" swap] && (X',B'+M') \arrow[dl, "h"] \\
				& Y
			\end{tikzcd}
		\end{center}
		Arguing as in the first paragraph of the proof of Lemma \ref{lem:properties_of_MMP}(i), we may further assume that $\big(X,(B+A)+M\big)$ is a $\Q$-g-pair. Then there exists a $\Q$-divisor $H$ on $X$ which is ample over $Y$ such that $ K_X+B+A+M+H \sim_{\Q,g} 0 $, and hence, by definition, the above MMP step is constructed as follows:
		$$ h \colon X' \simeq \proj_Y \Bigg( \bigoplus_{m \geq 0} g_* \OO_X(-mH) \Bigg) \longrightarrow Y. $$
		
		Let $H'$ be the strict transform of $H$ on $X'$. Then $-H'$ is ample over $Y$. Since the $\Q$-divisor $A$ on $X$ is ample over $Z$ and since relative ampleness is an open condition, we may find a $\Q$-divisor $C$ on $Y$ which is ample over $Z$ such that the divisor $ A - g^*C $ is still ample over $Z$. Additionally, for any sufficiently small rational number $ \varepsilon > 0 $, both $\Q$-divisors 
		\[ A-g^*C+\varepsilon H \quad \text{and} \quad h^*C-\varepsilon H' \quad \text{are ample over } Z . \]
		Indeed, the claim about the former follows as above, while the claim about the latter follows from Lemma \ref{lem:KM98_1.45_relative}. Thus, we may find a sufficiently general
		$$ 0 \leq  E_{\varepsilon} \sim_{\Q,Z} A-g^*C+\varepsilon H $$
		such that $ \big(X,(B+E_\varepsilon)+M \big)$ is an NQC lc g-pair and any lc center of $\big(X,(B+E_\varepsilon)+M\big)$ is an lc center of $\big(X,(B+A)+M\big)$. We denote by $E'_\varepsilon$ the strict transform of $E_\varepsilon$ on $X'$. Since the map $ \varphi \colon X \dashrightarrow X' $ is also a single step of a $(K_X+B+E_\varepsilon+M)$-MMP over $Z$ for $\varepsilon \ll 1$ due to the choice of $E_\varepsilon$, the NQC g-pair $\big(X',(B'+E'_\varepsilon)+M'\big)$ is also lc. Furthermore, we may choose
		$$ 0 \leq \widetilde A \sim_{\Q,Z} h^*C-\varepsilon H' $$
		sufficiently general such that the NQC g-pair $\big(X',(B'+E'_\varepsilon+ \widetilde A)+M'\big)$ is lc and $\widetilde A$ contains no lc center of $\big(X',(B'+E'_\varepsilon+\widetilde A)+M\big)$. Now, we set 
		$$ \widetilde B := B' + E'_\varepsilon , $$
		and by construction we have 
		$$ \varphi_*(B+A) \sim_{\Q,Z} \widetilde B + \widetilde A . $$ 
		Therefore, the divisors $ \widetilde A $ and $ \widetilde B $ on $X'$ satisfy all the requirements.
	\end{proof}
	
	\begin{thm}\label{thm:EGMM_boundary_contains_ample}
		Let $ \big( X/Z,(B+A)+M \big) $ be an NQC lc g-pair, where $ A $ is an effective $ \R $-Cartier $\R$-divisor which is ample over $ Z $. If the divisor $K_X+B+A+M$ is pseudo-effective over $Z$, then there exists a $(K_X+B+A+M)$-MMP over $Z$ which terminates with a good minimal model of $ \big( X,(B+A)+M \big) $ over $ Z $.
	\end{thm}
	
	\begin{proof}
		By \cite[Theorem 1.3(1)]{LX23a} the g-pair $\big(X, (B+A) + M \big)$ has a minimal model in the sense of Birkar-Shokurov over $ Z $. Therefore, by Proposition \ref{prop:mainprop} there exists a $ (K_X + B + A + M) $-MMP over $Z$ which terminates with a minimal model $ \big( X',(B'+A')+M' \big) $ of $ \big( X,(B+A)+M \big) $ over $ Z $. By replacing $A$ with a general member of its $\R$-linear system, we may assume that $A$ contains no lc center of the g-pair $ \big(X/Z,(B+A)+M \big)$. It follows now from Lemma \ref{lem:MMP_ampleness_boundary} and from \cite[Theorem 1.2]{Xie22} that the divisor $ K_{X'} + B' + A' + M' $ is semi-ample over $ Z $, which proves the assertion.
	\end{proof}
	
	\begin{cor}\label{cor:finite_generation}
		Let $ \big( X, (B+A)+M \big) $ be an lc $\Q$-g-pair with data $ W \to X \overset{\pi}{\longrightarrow} Z $ and $ M_W $, where $A$ is an effective $\Q$-Cartier $\Q$-divisor on $ X $ which is ample over $Z$. Then 
		\[ R(X/Z,K_X+B+A+M) := \bigoplus_{m \geq 0} \pi_* \OO_X \big( m(K_X+B+A+M) \big) \]
		is a finitely generated $ \OO_Z $-algebra.
	\end{cor}
	
	\begin{proof}
		Follows immediately from Theorem \ref{thm:EGMM_boundary_contains_ample}.
	\end{proof}

	\section{Applications -- Part II}
	\label{section:applications_II}
	
	In this section we generalize appropriately to the setting of g-pairs all results from \cite[Subsection 3.2]{Hash20d}.
	
	First, we extend \cite[Corollaries 3.8 and 3.9]{Hash20d} to the context of g-pairs as well as to the relative setting; see Corollaries \ref{cor:Hash20d_3.8} and \ref{cor:Hash20d_3.9}, respectively. For the definition of the notion of a \emph{relatively log abundant} log canonical generalized pair that appears in these two results we refer, for example, to \cite[Subsection 2.3]{Hash20d}. Before stating and proving Corollaries \ref{cor:Hash20d_3.8} and \ref{cor:Hash20d_3.9}, we note that the phrase \enquote{an MMP with scaling contains only finitely many (relatively) log abundant generalized pairs} used below means the following: if we have a $ (K_X+B+M) $-MMP over $Z$ with scaling of $ A $ and if for each $ i \geq 1 $ we denote by $(X_i,B_i+M_i)$ the g-pair appearing at the $ i $-th step of this MMP, where $ (X_1,B_1+M_1) := (X,B+M) $, then there are only finitely many indices $i$ such that $K_{X_i} + B_i + M_i$ is log abundant over $Z$ with respect to $(X_i,B_i+M_i)$. 
	
	\begin{cor}\label{cor:Hash20d_3.8}
		Any MMP with scaling of an ample divisor starting with an NQC $ \Q $-factorial dlt g-pair contains only finitely many log abundant NQC $ \Q $-factorial dlt g-pairs.
	\end{cor}
	
	\begin{proof}
		Fix an NQC $ \Q $-factorial dlt g-pair $(X/Z,B+M)$. Pick an effective $\R$-Cartier $\R$-divisor $A$ on $X$ which is ample over $ Z $ such that $ \big(X/Z, (B+A)+M \big) $ is lc and $ K_X+B+A+M $ is nef over $ Z $. We distinguish two cases.
		
		If $ K_X+B+M $ is not pseudo-effective over $ Z $, then by \cite[Lemma 4.4(1)]{BZ16} any $ (K_X+B+M) $-MMP over $ Z $ with scaling of $ A $ terminates with a Mori fiber space of $ (X,B+M) $ over $ Z $, so the statement clearly holds.
		
		If $K_X+B+M$ is pseudo-effective over $ Z $, then run a $ (K_X+B+M) $-MMP over $ Z $ with scaling of $ A $, denote by $ \lambda_i $ the nef thresholds at the steps of this MMP and set $ \lambda := \lim_{i \to +\infty} \lambda_i $. If $ \lambda > 0 $, then this MMP terminates; see the second paragraph of the proof of Lemma \ref{lem:Gon11_Thm5.1}. If $ \lambda = 0 $, then by \cite[Theorem 6.6]{LX23a}, cf.\ \cite[Theorem 3.15]{Hash22a}, this MMP contains only finitely many log abundant NQC $ \Q $-factorial dlt g-pairs. Thus, in any case, the statement holds.
	\end{proof}
	
	\begin{cor}\label{cor:Hash20d_3.9}
		Let $(X/Z,B+M)$ be an NQC lc g-pair and let $A$ be an effective $\R$-Cartier $\R$-divisor on $X$ which is ample over $ Z $ such that $ \big(X/Z, (B+A)+M \big) $ is lc and $ K_X+B+A+M $ is nef over $ Z $. Then there exists a $(K_X+B+M)$-MMP over $ Z $ with scaling of $A$ which contains only finitely many log abundant NQC lc g-pairs.
	\end{cor}
	
	\begin{proof}
		Assume first that $K_X+B+M$ is not pseudo-effective over $ Z $. By Theorem \ref{thm:mainthm} there exists a $(K_X+B+M)$-MMP over $ Z $ with scaling of $A$ that terminates with a Mori fiber space of $ (X,B+M) $ over $ Z $, so the statement clearly holds.
		
		Assume now that $K_X+B+M$ is  pseudo-effective over $ Z $. By \cite[Theorem 1.3(1)]{LX23a}, for any $ \nu \in (0,1] $ the NQC lc g-pair $\big(X, (B+\nu A)+M \big)$ has a minimal model in the sense of Birkar-Shokurov over $ Z $. Therefore, by Lemma \ref{lem:scaling_numbers_to_zero} we can construct a $(K_X+B+M)$-MMP over $ Z $ with scaling of $A$ whose nef thresholds $ \lambda_i $ converge to zero, that is, $ \lim_{i \to +\infty} \lambda_i = 0 $, regardless of whether this MMP terminates or not. If it does terminate, then the statement clearly holds. Otherwise, to prove the statement, we argue by contradiction as follows. Assume that this $(K_X+B+M)$-MMP over $ Z $ with scaling of $A$ contains infinitely many log abundant NQC lc g-pairs. By considering a dlt blow-up $ h \colon (X',B'+M') \to (X,B+M) $ and by applying Theorem \ref{thm:lifting_MMP_with_scaling}, we can construct a $(K_{X'}+B'+M')$-MMP over $Z$ with scaling of $ h^* A $, which contains infinitely many log abundant NQC lc g-pairs according to Lemma \ref{lem:dlt_blow-up_log_abundant} and whose nef thresholds $\mu_i$ converge to zero, that is, $\lim_{i \to +\infty} \mu_i = 0 $. However, this is impossible by \cite[Theorem 6.6]{LX23a}, cf.\ \cite[Theorem 3.15]{Hash22a}. Thus, the statement holds in this case as well.
	\end{proof}
	
	The above two results are closely related to the termination of flips conjecture. In fact, two applications of Corollary \ref{cor:Hash20d_3.8} in this direction will be provided below; see Lemma \ref{lem:Hash20d_3.10_auxiliary} and Corollary \ref{cor:Hash20d_3.10}. For further information we refer to the introduction of \cite{Hash20d}.
	
	\begin{conv}
		From this point forward we work exclusively in the absolute setting; see the beginning of Subsection \ref{subsection:log_abundant}.
	\end{conv}
	
	As another application of Theorem \ref{thm:MM_NQC_NZD}(ii), we may generalize \cite[Lemma 3.11]{Hash20d} to the context of g-pairs.
	
	\begin{lem}\label{lem:Hash20d_3.11}
		Let $(X,B+M)$ be an NQC lc g-pair such that $K_X+B+M$ is pseudo-effective. Assume that there exists a projective morphism $ \pi \colon X \to Y $ to a projective variety $ Y $ such that $ \dim Y \leq 4 $ and $K_X+B+M \sim_{\R,Y} 0$. Then $(X,B+M)$ has a minimal model.
	\end{lem}
	
	\begin{proof}
		By taking the Stein factorization of $\pi$, we may assume that $Y$ is a normal variety and that $\pi$ is a fibration.
		By the canonical bundle formula \cite{Fil20,HanLiu21} there exists an NQC lc g-pair structure $ (Y,B_Y+M_Y) $ on $ Y $ such that
		\begin{equation}\label{eq:CBF_dim4}
			K_X+B+M \sim_\R \pi^*(K_Y+B_Y+M_Y) .
		\end{equation}
		By Corollary \ref{cor:maincor_II}(i), $ (Y,B_Y+M_Y)$ has a minimal model, and hence it admits an NQC Nakayama-Zariski decomposition by Theorem \ref{thm:MM_NQC_NZD}(ii). By \eqref{eq:CBF_dim4} and by \cite[Remark 3.21]{Tsak21}, $ (X,B+M) $ admits an NQC Nakayama-Zariski decomposition as well, and thus it has a minimal model by Theorem \ref{thm:MM_NQC_NZD}(ii).
	\end{proof}
	
	The following two results constitute analogues of \cite[Corollaries 3.13 and 3.12]{Hash20d}, respectively, in the setting of g-pairs.
	
	\begin{cor}\label{cor:Hash20d_3.13}
		Let $(X,B+M)$ be an NQC lc g-pair. Assume that $K_X+B+M$ is pseudo-effective and abundant and that all lc centers of $(X,B+M)$ have dimension at most $ 4 $. Then $(X,B+M)$ has a minimal model which is abundant.
	\end{cor}
	
	\begin{proof}
		First, consider a dlt blow-up $ h \colon (X',B'+M') \to (X,B+M) $ of $(X,B+M)$, and set $ (X_1, B_1+M_1) := (X,B+M) $, $ (X_1', B_1'+M_1') := (X',B'+M') $ and $ h_1 := h$. As in the second paragraph of the proof of Corollary \ref{cor:Hash20d_3.9}, we can construct the following diagram:
		\begin{center}
			\begin{tikzcd}[column sep = 1em, row sep = large]
				(X_1',B_1'+M_1') \arrow[d, "h_1" swap] \arrow[rr, dashed, "\rho_1"] && (X_2',B_2'+M_2') \arrow[d, "h_2" swap] \arrow[rr, dashed, "\rho_2"] && (X_3',B_3'+M_3') \arrow[d, "h_3" swap] \arrow[rr, dashed, "\rho_3"] && \dots 
				\\ 
				(X_1,B_1+M_1) \arrow[rr, dashed, "\pi_1"] && (X_2,B_2+M_2) \arrow[rr, dashed, "\pi_2"] && (X_3,B_3+M_3) \arrow[rr, dashed, "\pi_3"] && \dots
			\end{tikzcd}
		\end{center}
		where the sequence at the bottom is a $(K_X+B+M)$-MMP with scaling of an ample divisor $H$ whose nef thresholds $ \lambda_i $ satisfy $ \lambda := \lim_{i \to +\infty} \lambda_i = 0 $, the sequence on top is a $(K_{X'}+B'+M')$-MMP with scaling of $ h^* H $ whose nef thresholds $ \mu_i $ also satisfy $ \mu := \lim_{i \to +\infty} \mu_i = 0 $, and each map $ h_i \colon (X_i',B_i'+M_i') \to (X_i,B_i+M_i) $ is a dlt blow-up.
		
		Next, for any $i \geq 1$ and any lc center $T_i$ of $(X'_i,B'_i+M'_i)$ we define an NQC dlt g-pair $(T_i,B'_{T_i}+M'_{T_i})$ by adjunction (see \cite[Subsection 2.3]{HanLi22} for the details): 
		\[ K_{T_i}+B'_{T_i}+M'_{T_i}=(K_{X'_i}+B'_{i}+M'_i)|_{T_{i}} . \]
		We will show that $(T_i,B'_{T_i}+M'_{T_i})$ has a minimal model or a Mori fiber space. Since all lc centers of $(X,B+M)$ have dimension at most $ 4 $ by assumption, the same holds for all lc centers of $(X_i,B_i+M_i)$ as well. Therefore, the morphism $h_i \colon X'_i \to X_i$ induces a morphism $T_i \to S_i$ to a projective variety $S_i$ such that $ \dim S_i \leq 4$ and $K_{T_i}+B'_{T_i}+M'_{T_i} \sim_{\R,S_i} 0$; see the second paragraph of the proof of Lemma \ref{lem:dlt_blow-up_log_abundant} for the details. By Lemma \ref{lem:Hash20d_3.11} or by Proposition \ref{prop:mainprop} the g-pair $(T_i,B'_{T_i}+M'_{T_i})$ has indeed a minimal model or a Mori fiber space.
		
		Due to the above fact, by using the standard argument of special termination (see \cite{LMT23} and \cite[Remark 2.21]{Hash20d} for the details), together with \cite[Theorem 4.1]{HanLi22}, we may find a positive integer $m$ such that any step in the sequence
		$$(X'_m,B'_m+M'_m) \dashrightarrow (X'_{m+1},B'_{m+1}+M'_{m+1}) \dashrightarrow\cdots \dashrightarrow (X'_i,B'_i+M'_i)\dashrightarrow \cdots$$
		is an isomorphism in a neighborhood of $\Supp \llcorner B'_{i}\lrcorner$.
		Since $\mu=0$, by \cite[Lemma 3.8]{Hash22a} (see also \cite[Remark III.2.8 and Lemma V.1.9]{Nak04}) the restriction of 
		$K_{X'_m}+B'_m+M'_m$ to any component of $\llcorner B'_m\lrcorner$ is nef and, additionally, for any divisorial valuation $P$ over $X'_m$ whose center $c_{X'_m}(P)$ intersects an lc center of $(X'_m,B'_m+M'_m)$ we have
		\[ \sigma_P (K_{X'_m}+B'_m+M'_m)=0 . \]
		Since $K_{X}+B+M$ is abundant and since the map $(X',B'+M') \dashrightarrow (X_m',B_m'+M_m')$ is a partial $(K_{X'}+B'+M')$-MMP, by the Negativity lemma and by \cite[Remark 2.15(2)]{Hash20d} we deduce that $K_{X'_m}+B'_m+M'_m$ is also abundant. Hence, the g-pair $(X'_m,B'_m+M'_m)$ has a minimal model in the sense of Birkar-Shokurov by \cite[Theorem 3.14]{Hash22a}, and thus so does the g-pair $ (X_1', B_1'+M_1') = (X',B'+M') $; indeed, this follows by applying Proposition \ref{prop:mainprop} to the g-pair $(X'_m,B'_m+M'_m)$ and by observing that we obtain overall a $(K_{X'} + B' + M')$-MMP which terminates. Finally, according to Remarks \ref{rem:EMM_reduction} and \ref{rem:abundant_MM}, the given g-pair $(X,B+M)$ has a minimal model which is abundant, as claimed.
	\end{proof}
	
	\begin{cor}\label{cor:Hash20d_3.12}
		Let $(X,B+M)$ be an NQC lc g-pair of dimension $\dim X = 6$. Assume that $K_X+B+M$ is pseudo-effective and abundant and that $\llcorner B \lrcorner = 0$. Then $(X,B+M)$ has a minimal model which is abundant.
	\end{cor}
	
	\begin{proof}
		Since $\dim X = 6$, the condition $\llcorner B \lrcorner =0$ implies that any lc center of $(X,B+M)$ has dimension $\leq 4$, so the statement follows immediately from Corollary \ref{cor:Hash20d_3.13}.
	\end{proof}
	
	We stress that Corollaries \ref{cor:Hash20d_3.13} and \ref{cor:Hash20d_3.12} are not exact analogues of \cite[Corollaries 3.13 and 3.12]{Hash20d}, respectively, in the context of g-pairs and we explain now the main reason why we cannot obtain completely analogous statements in this more general framework. In \cite[Corollaries 3.12 and 3.13]{Hash20d} it is shown that the divisor $K_X+\Delta$ is abundant using the assumption that $ \kappa_\iota(X,K_X+\Delta)\geq \dim X-3 $ as follows: first, by taking the Iitaka fibration $f \colon X\to V$ associated with $K_X+\Delta$ and by invoking the Abundance theorem for 3-dimensional lc pairs, one deduces that
	$$ \kappa_\iota(F,K_F+\Delta_F)=\kappa_\sigma(F,K_F+\Delta_F)=0 , $$
	where $F$ is a general fiber of $f$ and $(K_F+\Delta_F) :=(K_X+\Delta)|_F$, and then by applying \cite[Proposition V.2.7(9)]{Nak04} one obtains
	$$ \kappa_\sigma(X,K_X+\Delta) \leq \kappa_\sigma(F,K_F+\Delta_F)+\dim V = \kappa_\iota(X,K_X+\Delta) , $$
	which implies that $K_X+\Delta$ is abundant. However, the Abundance conjecture fails even in dimension one for generalized pairs; see, for instance, \cite[Section 3]{BH14b}. The following example, whose first part is briefly discussed also in \cite[Section 2, Paragraph after Theorem 2.1]{Tot09}, shows that the canonical divisor $K_X+B+M$ of a g-pair $(X,B+M)$ may not be abundant even if it is nef and satisfies $\kappa_\iota(X,K_X+B+M)=0$.
	
	\begin{exa}
		Let $g \colon X \to \mathbb{P}^2$ be the blow-up of $\mathbb{P}^2$ at nine very general points of an elliptic curve $E$ and let $B$ be the strict transform of $E$ on $X$. Since $E \in |\OO_{\mathbb{P}^2}(3)| $, and thus $E^2 = 9$, it is easy to see that $B^2 = 0$, and since $B$ is irreducible, we infer that $B$ is nef. Observe also that $B$ is not numerically trivial, and since it is not big either, we conclude that $\kappa_\sigma(X,B) = 1$. Furthermore, the normal bundle $\mathcal{O}_X(B)|_B$ of $B$ in $X$ is a non-torsion line bundle of degree $0$ on $B$, and hence no positive multiple of $B$ moves in $X$, that is, for every $n \geq 1$, the effective divisor $nB$ is the unique element of the linear system $ |nB|$, which yields $\kappa(X,B) = 0$.
		
		Set $M:=B$ and note that $(X,B+M)$ is an lc g-pair. By construction it holds that $K_X+B \sim g^* (K_{\mathbb{P}^2} + E) \sim 0$, so
		\[ K_X+B+M \sim B \; \text{ is nef} \]
		and by the above we obtain
		\[ \kappa_\iota(X,K_X+B+M)=\kappa(X,B)=0 , \]
		while
		\[ \kappa_\sigma(X,K_X+B+M)=\kappa_\sigma(X,B)=1 . \]
	\end{exa}
	
	It remains to deduce an analogue of \cite[Corollary 3.10]{Hash20d} in the setting of g-pairs; this is Corollary \ref{cor:Hash20d_3.10} below. We begin with an auxiliary result, which plays a key role in the proof of Corollary \ref{cor:Hash20d_3.10}. Specifically, the next lemma provides a sufficient condition for the termination of an MMP with scaling of an ample divisor starting from an NQC $\Q$-factorial dlt g-pair, essentially when the corresponding nef thresholds converge to zero.
	
	\begin{lem}\label{lem:Hash20d_3.10_auxiliary}
		Let $(X,B+M)$ be an NQC $\Q$-factorial dlt g-pair. Assume that $K_X+B+M$ is pseudo-effective and log abundant with respect to $(X,B+M)$, and that the stable base locus of $K_X+B+M$ does not contain the center of any divisorial valuation $ P $ over $X$ such that $ a(P,X,B+M) < 0 $. Then any $(K_X+B+M)$-MMP with scaling of an ample divisor terminates with a minimal model of $(X,B+M)$ which is log abundant.
	\end{lem}
	
	\begin{proof}
		Run a $(K_X+B+M)$-MMP with scaling of an ample divisor:
		$$ (X_1, B_1+M_1) := (X,B+M) \dashrightarrow (X_2,B_2+M_2) \dashrightarrow \cdots \dashrightarrow (X_i,B_i+M_i)\dashrightarrow \cdots $$
		Fix $i\geq 1$ and pick an lc center $S_i$ of $(X_i,B_i+M_i)$. Then there exists an lc center $S$ of $(X,B+M)$ such that the map $X \dashrightarrow X_i$ induces a birational map $S \dashrightarrow S_i$; see \cite[Lemma 2.18(iii)]{Tsak21}. Define NQC dlt g-pairs $(S,B_S+M_S)$ and $(S_i,B_{S_i}+M_{S_i})$ by adjunction: 
		\[ K_S+B_S+M_S=(K_X+B+M)|_S \quad \text{and} \quad K_{S_i}+B_{S_i}+M_{S_i}=(K_{X_i}+B_i+M_i)|_{S_i} ; \]
		see \cite[Subsection 2.3]{HanLi22} for the details. By \cite[Lemma 2.18(iv)]{Tsak21} for any divisorial valuation $P$ over $S$ we have
		\[ a(P,S,B_S+M_S) \leq a(P,S_i,B_{S_i}+M_{S_i}) . \]
		In addition, by our assumptions, \cite[Lemma 3.5]{Hash22a} and \cite[Lemma 3.7(4)]{LX23a}, for every prime divisor $G$ on $S_i$ we have 
		\[ a(G,S_i,B_{S_i}+M_{S_i})\leq a(G,S,B_S+M_S) . \]
		Therefore, \cite[Lemma 3.9]{Hash22a} implies that the divisor $K_{S_i}+B_{S_i}+M_{S_i}$ is abundant. Hence, for every $i \geq 1$ the g-pair $(X_i,B_i+M_i)$ is log abundant. It follows now from Corollary \ref{cor:Hash20d_3.8} that the above $(K_X+B+M)$-MMP with scaling terminates with a minimal model $(Y,B_Y+M_Y)$ of $(X,B+M)$, while by construction the divisor $K_Y+B_Y+M_Y$ is log abundant with respect to $(Y,B_Y+M_Y)$.
	\end{proof}
	
	\begin{cor}\label{cor:Hash20d_3.10}
		Let $(X,B+M)$ be an NQC lc g-pair. Assume that $K_X+B+M$ is pseudo-effective and log abundant with respect to $(X,B+M)$, and that the stable base locus of $K_X+B+M$ does not contain the center of any divisorial valuation $P$ over $X$ such that $ a(P,X,B+M) < 0 $. Then $(X,B+M)$ has a minimal model which is log abundant.
	\end{cor}
	
	\begin{proof}
		Let $ h \colon (X',B'+M') \to (X,B+M) $ be a dlt blow-up of $(X,B+M)$. According to Lemma \ref{lem:dlt_blow-up_log_abundant}, the divisor $K_{X'}+B'+M'$ is pseudo-effective and log abundant with respect to $(X',B'+M')$, and by \cite[Lemma 2.3]{LMT23} we deduce that the stable base loci of $K_X+B+M$ and $K_{X'}+B'+M'$ are related as follows:
		\[ \sbs(K_{X'}+B'+M') = h^{-1} \big( \sbs(K_X+B+M) \big) . \]
		Therefore, $(X',B'+M')$ satisfies the same hypotheses as $(X,B+M)$, and now Lemma \ref{lem:Hash20d_3.10_auxiliary} implies that it has a minimal model $(Y'',B_Y''+M_Y'')$ which is log abundant. It follows from Remark \ref{rem:EMM_reduction} that $(X,B+M)$ has a minimal model $(Y,B_Y+M_Y)$. Let $ t \colon (Y',B_Y'+M_Y') \to (Y,B_Y+M_Y) $ be a dlt blow-up of $(Y,B_Y+M_Y)$. Note that $(Y',B_Y'+M_Y')$ is a minimal model in the sense of Birkar-Shokurov of $(X,B+M)$, so it is also log abundant according to Lemma \ref{lem:log_abundant_MMs}. By Lemma \ref{lem:dlt_blow-up_log_abundant} we now conclude that $(Y,B_Y+M_Y)$ is log abundant, which completes the proof.
	\end{proof}
	
	As in the case of Corollaries \ref{cor:Hash20d_3.13} and \ref{cor:Hash20d_3.12}, we remark that Corollary \ref{cor:Hash20d_3.10} is also only a partial analogue of \cite[Corollary 3.10]{Hash20d} in the context of g-pairs. Indeed, in contrast to the case of usual pairs, where it is well-known that a nef and log abundant canonical divisor is semi-ample (see, for instance, \cite{Xie22} for more information), we cannot necessarily deduce the existence of a good minimal model of the g-pair $(X,B+M)$ from Corollary \ref{cor:Hash20d_3.10}, as demonstrated by \cite[Example 2.2]{LX23b}.

	\bibliographystyle{amsalpha}
	\bibliography{BibliographyForPapers}
	
\end{document}